\documentclass[10pt]{article}

\usepackage{amssymb,latexsym, amsmath} 
\usepackage{amscd} 

\usepackage{amssymb}
\usepackage{amsthm}
\usepackage{amsmath}
\usepackage{amsfonts}
\usepackage{rotating}
\usepackage{mathabx}

\usepackage{lscape}

\newtheorem{Cor}{Corollary}
\newtheorem{definition}{Definition}
\newtheorem{proposition}{Proposition}
\newtheorem{theorem}{Theorem}

\newcommand\Tr{\mathrm{Tr}}

\newcommand\ad{\mathrm{ad}}
\newcommand\Ad{\mathrm{Ad}}
\newcommand\goth{\mathfrak}
\newcommand\R{\mathbb R}
\newcommand\Z{\mathbb Z}
\newcommand\ind{\mathrm{ind}\,}
\newcommand\cork{\mathrm{corank }\,}
\newcommand\Ker{\mathrm{Ker}\, }
\newcommand\Ann{\mathrm{Ann}\, }

\newcommand\rank{\mathrm{rank}\, }
\newcommand\corank{\mathrm{corank}\, }
\newcommand\codim{\mathrm{codim}\, }
\newcommand\Sing{{\mathsf{Sing}}}

\title{Jordan--Kronecker invariants of finite-dimensional Lie algebras}
\author{Alexey Bolsinov and Pumei Zhang}

\begin{document}
\maketitle

\section{Motivation}\label{motivation}
\begin{itemize}

\item  A Lie algebra $\goth g$ is defined by its structure tensor $c_{ij}^k$. The invariants of $\goth g$ are, in essence, those of $c_{ij}^k$.  This tensor is quite complicated to study and it is natural to try somehow to simplify it first. The classical method is to consider, instead of this tensor, a simpler object, namely, the operator 
$\ad_\xi=\Bigl(\sum c_{ij}^k \xi^i\Bigr)$ for a generic vector $\xi\in\goth g$. This operator defines the decomposition of  $\goth g$ into (generalised) eigenspaces: the zero eigenspace is known as a Cartan subalgebra, the other eigenspaces are root subspaces. Using this approach systematically leads, in particular, to the classification of semisimple Lie algebras. 

We are going to do a similar thing, but instead of the operator  $\ad_\xi$, we suggest to consider the bilinear form  $\mathcal A_x=\Bigl( \sum c_{ij}^k x_k \Bigr)$ for a regular covector $x\in\goth g^*$.  This form  does not give any non-trivial invariants (except for its corank called the {\it index} of $\goth g$).
However, non-trivial invariants immediately appear as soon as we consider a pair of forms
 $\mathcal A_x$ and $\mathcal A_a$ for $x,a\in \goth g^*$.   From the algebraic viewpoint these invariants  look quite natural, and their systematic analysis seems to be an interesting mathematical problem.
 
\item  Some already known results become more transparent  and receive a new interpretation if we look at them from the viewpoint of Jordan--Kronecker invariants.
Besides useful reformulations,  in this way one can get  new non-trivial results (for example,  Theorems \ref{Vorontsov} and \ref{thfrob} below).

\item We expect that these technics  will be useful in the study of the coadjoint representation of non-semisimple Lie algebras. Many papers are focused  just on the semisimple case, but the methods used in this area are so specific that their generalisation to the case of arbitrary Lie algebras is hardly possible.  It would be very desirable to develop universal tools and ideas.

\item Finally, the main reason why we have been involved in this area is the generalised ``argument shift conjecture'' discussed below.
Apparently,  to prove or disprove  it will necessarily require the concept of Jordan-Kronecker invariants.
This conjecture itself seems to be very important as the argument shift method  is one of few indeed universal constructions which are worth being treated in detail.

\end{itemize}

\section{Some historical remarks}
The idea of Jordan--Kronecker invariants is based on the results, methods and constructions  invented and developed by different mathematicians in different years  and sometimes even not related to each other.

\begin{enumerate}
\item  The main point for us is, no doubt, the argument shift method suggested in 1976 by A.S.\,Mischenko and A.T.\,Fomenko \cite{MischFom}.  This construction  has  been analysed, developed and generalised by participants of the seminar ``Modern geometric methods''  at Moscow State University in the 80s  (V.V.\,Trofimov, A.V.\,Brailov, Dao Trong Tkhi, M.V.\,Mescherjakov and others) and many of their results have been extremely important.

\item In the late 80s,  I.M.\,Gelfand and I.\,Zakharevich discovered a very interesting relationship between compatible Poisson brackets, veronese webs  and the Jordan--Kronecker decomposition theorem for a pair of skew-symmetric forms.
This observation then played a very important role in a series or papers by I.M.\,Gelfand and I.\,Zakharevich  \cite{gelzak, zak} devoted to Kronecker pencils and their applications to the theory of integrable systems.

\item 
The Jordan--Kronecker decomposition theorem in full generality is presented in the paper  \cite{thom} by R.\,Thompson.  Although  all essential ingredients of this theorem can be found in classical works by C.\,Jordan and L.\,Kronecker, to the best of our knowledge,  the paper \cite{thom} is the first to contain a rigorous formulation and proof of this purely algebraic result\footnote{Yu.\,Neretin has recently informed us about the work by G.B.\,Gurevich \cite{gurevich} containing the same result, but we have not had a chance to  see this paper.  We hope to do it in the near future and will then revise our main reference.}.

\item In the symplectic case, a transition from the algebraic canonical form of a pair of  skew-symmetric matrices to the differential-geometric normal form of a pair of compatible Poisson structures  has been carried out by
F.-J.\,Turiel   \cite{turiel}. That was a crucial step in understanding local structure of compatible Poisson structures.  However, the description of their normal forms  in the general case still remains an open and very difficult problem, see \cite{tur2}, \cite{tur3} for recent development in this area.

\item
In fact, the concept of Jordan--Kronecker invariants in implicit form can be found  in many papers devoted to integrable systems on Lie algebras.  Besides the above mentioned papers, first of all we would like to refer to the series of papers by A.\,Panasyuk \cite{Pan1, Pan2, Pan3} where the Jordan--Kronecker decomposition has been effectively used, see also \cite{BO-rcd}, \cite{Vorontsov}, \cite{koz}, \cite{Izosimov1}.

\end{enumerate}

Although all these ideas based on the Jordan--Kronecker decomposition seem to be very useful, they still remain widely unknown.
The present paper can be considered as an attempt to summarise  them in a unified and systematic way by putting into focus the Jordan-Kronecker invariants as a very natural algebraic object. Of course, the paper contains a number of new results too.

The structure of the paper is as follows.  Sections 3, 4, 5  can be viewed as introduction to the main  subject of the paper.  In Section 3, we recall some basic notions  and notation to be used throughout the paper.  Section 4 is devoted to the argument shift method,  Mischenko--Fomenko conjecture and its generalisation which we consider as the main motivation for our work. In Section 5,  we formulate the Jordan--Kronecker decomposition theorem for a pair of skew-symmetric forms and discuss some  linear algebraic corollaries from this result. These quite elementary facts will then be ``translated'' into the language of Lie algebras and will lead us (surprisingly  easily) to some not at all obvious results. 

This programme will be realised in Sections 7-10 in the context of Jordan--Kronecker invariants which are introduced in  Section 6.   The final section  is devoted to examples and computations.

The authors are very grateful to Andriy Panasyuk, Francisco-Javier Turiel and Ilya Zakharevich for very stimulating discussions. We also would like to thank the participants of an informal seminar which has been working over the past several years between Loughborough and Moscow, especially,  Andrey Oshemkov,  Sasha Vorontsov,  Andrey Konjaev, Anton Izosimov and Ivan Kozlov.   In many respects, the present paper is a result of these discussions.
The work was supported by the Ministry of Science and Education of Russia,  grants no.
14.740.11.0876 and 11.G34.31.0039.

\medskip
\section{Background:  basic notions, facts and notation}

Here we recall some basic notions and introduce notation we use throughout the paper. In what follows, we consider  vector spaces, Lie algebras and other algebraic objects over $\mathbb C$ unless otherwise specified.  The transition to the real case is usually straightforward.

\begin{itemize}

\item Finite-dimensional Lie algebra $\goth g$ and its dual space $\goth g^*$.

\item {\it Adjoint and coadjoint representations} of a Lie group $G$ and its Lie algebra~$\goth g$:
$$
\Ad_X \xi = \left. \frac{d}{dt}\right|_{t=0} X \exp (t\xi) X^{-1},
$$
$$
\langle \Ad^*_X  a , \xi\rangle = \langle a , \Ad^{-1}_X \xi \rangle,
$$ 
where $X\in G$, $\xi\in\goth g$, $a\in\goth g^*$. Similarly:
$$
\ad_\xi \eta=[\xi,\eta], \quad  \langle  \ad^*_\xi a , \eta\rangle =\langle a, -[\xi,\eta]\rangle.
$$

If  $\goth g\subset \mathrm{gl}(n,\mathbb C)$ is a matrix Lie algebra, then the coadjoint representation can be defined explicitly, for example,  as follows. By using the pairing  $\langle a,\xi\rangle = \mathrm{Tr}\, a\xi$,  we identify $\goth g^*$ with the subspace $\bar{\goth g}^\top\subset \mathrm{gl}(n,\mathbb C)$,  obtained from $\goth g$ by transposition and complex conjugation. Then
$$
\Ad^*_X a = \mathrm{pr} (XaX^{-1}), \quad \ad^*_\xi a = \mathrm{pr} ([\xi, a]),
$$
where $\mathrm{pr}:  \mathrm{gl}(n,\mathbb C) \to \goth g^*=\bar{\goth g}^\top$ is the natural projection with the kernel $\goth g^\bot$, i.e.,  $\langle a - \mathrm{pr}(a) , \goth g\rangle = 0$ for every $a\in\mathrm{gl}(n,\mathbb C)$.

\item The {\it Lie--Poisson bracket} on $\goth g^*$:
$$
\{f,g\}(x) = \langle x , [df(x),dg(x)] \rangle, \qquad  x\in\goth g^*,  \quad f,g:\goth g^* \to \mathbb C.
$$

The corresponding Poisson tensor is given by the skew-symmetric matrix $\mathcal A_x = \Bigl(c_{ij}^k x_k\Bigr)$, i.~e., depends linearly on coordinates.

The algebra $P(\goth g)$ of polynomials on $\goth g^*$ endowed with this bracket is called the {\it Lie-Poisson algebra} (associated to $\goth g$).

\item The coadjoint orbits are symplectic leaves of the Lie-Poisson structure, and vice versa.

The {\it Casimir functions} (i.e.,  functions $f$ satisfying  $\{f,g\}=0$ for all $g$) are exactly the invariants of the coadjoint representation. We shall denote the set (algebra) of coadjoint invariants  by $I(\goth g)$. We do not specify here the class of such functions (polynomial, rational, etc.), because in general we can only guarantee existence of locally analytic Casimir functions in a neighborhood of a generic point.  But even local Casimirs will be sufficient for our purposes.

\item The {\it annihilator} of an element $a\in \goth g^*$ is the stationary subalgebra of $a$ in the sense of the coadjoint representation:
$$
\mathrm{Ann}\, a =\{ \xi\in\goth g~|~ \ad^*_\xi a=0\}.
$$

In terms of the Lie-Poisson structure, the annihilator of $a\in \goth g^*$ can be characterised as the kernel of the form  $\mathcal A_a$.
If $a\in\goth g^*$ is regular, then the differentials of (local) coadjoint invariants $f_i$,  form a basis of
 $\Ann a$. In general,  $df(a)$'s span a certain subspace in $\Ann a$.

\item The {\it index} of a Lie algebra $\goth g$ is the codimension of a regular coadjoint orbit. Equivalently,
$$
\ind \goth g= \min_{x\in\goth g^*} \dim\Ann\, x
$$

The index can also be characterised as the number of functionally independent (local) coadjoint invariants, i.e., Casimirs.

If $\ind \goth g=0$, then the Lie algebra  $\goth g$ is said to be {\it Frobenius}.

\item The {\it singular set}
$\Sing \subset \mathfrak g^*$ consists of those points $y \in \mathfrak g^*$ for which
$\cork\mathcal A_y>\ind\mathfrak g$,  where $\mathcal A_y$ is the Lie-Poisson tensor at the point 
$y$.  In other words, $\Sing$ is the set of all coadjoint orbits of non-maximal dimension.
Equivalently,
$$
\Sing=\{ y\in \goth g^*~|~ \dim \mathrm{Ann}\,(y) > \ind\goth g\}
$$

\item Let $f: \goth g^* \to \mathbb C$ be a coadjoint invariant,  $f\in I(\goth g)$.  Choose and fix a regular element $a \in\goth g^*$ and consider the functions of the form $f_\lambda (x) = f(x+\lambda a)$, $\lambda \in \mathbb C$. The family of functions
\begin{equation}
 \{ f(x+\lambda a) ~|~ f\in I(\goth g), \ \lambda\in \mathbb C\}
\label{classshifts}
\end{equation}
is said to be a {\it family of shifts  (of coadjoint invariants)}.  This classical definition from \cite{MischFom}  needs however to be slightly modified. The reason is that for non-algebraic Lie algebras the coadjoint invariants may not be globally defined, whereas we want to have a global and universal construction for all types of Lie algebras.  

Consider locally analytic  invariants $f_1,\dots, f_s$, $s=\ind \goth g$ defined in a neighbourhood of $a\in \goth g^*$  such that their differentials $df_i(a)$ form a basis of $\Ann a$ (recall that $a$ is regular so that such invariants do exist).  Take the Taylor expansions of $f_i$ at $a$:
$$
f_i(a+\lambda x) = f_i^0 + \lambda f_i^1(x) +\lambda^2 f_i^2(x) +\lambda^3 f_i^3(x) + \dots
$$
where $f_i^k(x)$ is a homogeneous polynomial in $x$ of degree $k$.

It is not hard to see that the collection of $f_i^k$'s is somehow equivalent to \eqref{classshifts}: in the simplest case, for example, when $f_i$ are homogeneous polynomials,   $f_i^k$'s form a spanning set of the family of shifts $f_i(x+\lambda a)$.
That is why, in what follows,  we replace \eqref{classshifts} by the subalgebra  $\mathcal F_a \subset P(\goth g)$ generated by the homogeneous polynomials
\begin{equation}
\label{generators}
f_i^k (x),  \quad  i=1,\dots,\ind \goth g,  \ k >0.
\end{equation} 
We call $\mathcal F_a$  the {\it algebra of (polynomial) shifts}.   Such a modification is useful  for at least three reasons (see \cite{bz}):
\begin{itemize}
\item this approach is universal and purely algebraic which allows us to work with arbitrary Lie algebras over any field of characteristic zero;
\item the  algebra of polynomial shifts $\mathcal F_a$ is canonical, i.e., does not depend on the choice of local invariants $f_1,\dots, f_s$ we started with;
\item  to construct the family \eqref{classshifts} of classical shifts, we need  to find the $\Ad^*$-invariants which is not an easy task, whereas  generating elements \eqref{generators} of  $\mathcal F_a$ can be found explicitly by solving systems of linear equations.  
\end{itemize}

\end{itemize}

\section{Generalised argument shift conjecture}

In the skew-symmetric forms $\mathcal A_x$ and $\mathcal A_a$ mentioned in Section \ref{motivation} one can easily recognise two well-known Poisson structures on the dual space   $\goth g^*$ of a finite-dimensional Lie algebra $\goth g$.

The first of them is the standard Lie-Poisson bracket:
\begin{equation}
\{ f, g \} (x) = \mathcal A_x \bigl(df(x), dg(x)\bigr)= \sum c_{ij}^k x_k \frac{\partial f}{\partial x_i} \frac{\partial g}{\partial x_j}, 
\label{LiePoisson}
\end{equation}
where $x\in \goth g^*, \ f,g:\goth g^* \to \mathbb C$.

From the algebraic viewpoint, a completely integrable system on $\goth g^*$ is a complete commutative family (subalgebra) $\mathcal F\subset P(\goth g)$. {\it Completeness} in this context means that $\mathcal F$ contains $\frac{1}{2} (\dim \goth g +\ind\goth g)$ algebraically independent polynomials.

One of the most efficient methods for constructing such families $\mathcal F\subset P(\goth g)$ is to use an additional Poisson structure compatible with  \eqref{LiePoisson}.  As the simplest structure with such a property, one can take the constant Poisson bracket given by the following well-known formula:
\begin{equation}
\{ f, g \}_a (x) = \mathcal A_a \bigl(df(x), dg(x)\bigr)= 
\sum c_{ij}^k a_k \frac{\partial f}{\partial x_i} \frac{\partial g}{\partial x_j},
\label{abracket}
\end{equation}
where $a\in \goth g^*$ is a fixed element.  Here we assume  $a\in\goth g^*$ to be regular although this formula makes sense for an arbitrary $a$, 

The argument shift method suggested by A.S.Mischenko and  A.T.Fomenko in \cite{MischFom} is based on the following observation (which can be naturally generalised to the case of arbitrary compatible Poisson brackets).  Let $f$ and $g$ be coadjoint invariants. Then the functions $f(x+\lambda a)$ and $g(x+\mu a)$ commute with respect to the both brackets \eqref{LiePoisson}  and \eqref{abracket}.
Notice that the shifts $f(x+\lambda a)$ are exactly  Casimirs for the linear combination $\{~,~\} + \lambda\{~,~\}_a$.  Replacing these  shifts, as was explained in Section 3, by the algebra $\mathcal F_a$ of polynomial shifts, we can reformulate the main result of \cite{MischFom} as follows.  

\begin{theorem}[A.S.\,Mischenko, A.T.\,Fomenko \cite{MischFom}]\label{MFth}\quad

{\rm 1)} The functions from $\mathcal F_a$ pairwise commute with respect to the both brackets  $\{~,~\}$ and $\{~,~\}_a$.

{\rm 2)} If $\goth g$ is semisimple,  then $\mathcal F_a$ is complete, i.e. contains $\frac12(\dim \goth g + \ind\goth g)$ algebraically independent polynomials.  
\end{theorem}

Although in general $\mathcal F_a$ is not necessarily complete,   A.S.\,Mischenko and A.T.\,Fomenko stated the following well known conjecture

\medskip

\noindent {\bf  Mischenko--Fomenko conjecture.}
{\it On the dual space $\goth g^*$ of an arbitrary Lie algebra $\goth g$ there exists a complete family $\mathcal F$ of commuting polynomials. }
\medskip

In other words, for each $\goth g$ one can construct a completely integrable (polynomial) system on $\goth g^*$ or,  speaking in algebraic terms,  the Lie-Poisson algebra $P(\goth g)$ always contains a complete commutative subalgebra.

This conjecture was proved in 2004 by  S.T.Sadetov \cite{sadetov}, see also \cite{bolsinov},\cite{vinyak}.  However, Sadetov's family $\mathcal F\subset P(\goth g)$ is essentially different from the family of shifts $\mathcal F_a$. Thus,  it is still an open question whether or not one can modify the argument shift method to construct a complete family in bi-involution.  In all the examples we know,  the answer is positive which allows us to propose the following bi-Hamiltonian version of the Mischenko--Fomenko conjecture.

\medskip
\noindent {\bf  Generalised argument shift conjecture.}   
{\it On the dual space $\goth g^*$ of an arbitrary Lie algebra $\goth g$ there exists a complete family $\mathcal G_a$ of  polynomials in bi-involution, i.e. in  involution w.r.t. the two brackets 
$\{~,~\}$ and $\{~,~\}_a$}.

\medskip

In fact,  our conjecture can be reformulated in the following equivalent way:
{\it the algebra  $\mathcal F_a$ of polynomial shifts can always be extended up to a complete subalgebra $\mathcal G_a\subset P(\goth g)$ of polynomials in bi-involution}.

\section{Jordan--Kronecker decomposition theorem}

The below theorem gives the classification of pairs of skew-symmetric forms
  $\mathcal A, \mathcal B$ by reducing them simultaneously to an elegant canonical block-diagonal form.
  
Usually one refers to this result as the Jordan--Kronecker theorem since the classical works by these two famous mathematicians  (written in the second half of the XIXth century) contain all of the most important ideas and ingredients of this construction.   A more recent reference is a very interesting paper by R.~Thompson, which  serves as a good and complete survey on this subject and related topics (see also \cite{gurevich} by G.B.\,Gurevich and  a note  \cite{koz} by I.~Kozlov with a  short proof).

\begin{theorem}
\label{JKD}
Let $\mathcal A$  and  $\mathcal B$ be two skew-symmetric bilinear forms on a complex vector space $V$.
Then by an appropriate choice of a basis, their matrices can be simultaneously reduced to the
following canonical block-diagonal form:
$$\mathcal A  \mapsto \begin{pmatrix}
\mathcal A_1 &   &    &  \\
  &  \mathcal A_2  & & \\
  & & \ddots & \\
  & & & \mathcal  A_k
  \end{pmatrix},
  \qquad
\mathcal B \mapsto \begin{pmatrix}
\mathcal B_1 &   &    &  \\
  &  \mathcal B_2  & & \\
  & & \ddots & \\
  & & &  \mathcal B_k
  \end{pmatrix}
  $$
where the pairs of the corresponding blocks $\mathcal A_i$ and $\mathcal B_i$ can be of the following three types:

\vskip-10pt
$$\hskip-30pt
\begin{array}{lcc}
  &  \mathcal A_i  &  \mathcal B_i     \\  & & \\
\begin{array}{l} \text{Jordan block } \\ (\lambda_i\in \mathbb{C}) \end{array} &
\begin{pmatrix}
     &   \!\!\! J(\lambda_i) \\  & \\
   \!  -J^\top(\lambda_i) &
 \end{pmatrix} &
 \begin{pmatrix}
   & \ \ -\mathrm{Id} \ \  \\  &  \\
   \mathrm{Id} &
 \end{pmatrix} \\     & &
 \\

\begin{array}{l}\text{Jordan block } \\ (\lambda_i=\infty) \end{array} &
  \begin{pmatrix}
   & \ \ - \mathrm{Id} \ \  \\   &  \\
   \mathrm{Id} &
 \end{pmatrix}
   &
\begin{pmatrix}
     &   \!\!\!\! J(0) \\    &  \\
   \!  -J^\top(0) &
 \end{pmatrix}
 \\  &  &
  \\

\begin{array}{l} \text{Kronecker} \\
\text{block} \end{array} &
 \!\!\!\!\!\!\!\!\!\!\!\! \begin{pmatrix}
     &
\hskip-10pt \boxed{ \begin{matrix}
  1 & 0  & & \\
      &\ddots & \ddots & \\
     &           &  1 & 0
     \end{matrix}}
     \\
\boxed{\begin{matrix}
\!\!\! -1  &             &   \\
  0 & \ddots &   \\
     & \ddots &\!\!\! -1 \\
     &             &0
  \end{matrix}
  }\end{pmatrix}
 &
  \begin{pmatrix}
     &
\hskip-10pt \boxed{ \begin{matrix}
  0 & 1  & & \\
      &\ddots & \ddots & \\
     &           &  0 & 1
     \end{matrix}}
     \\
\boxed{\begin{matrix}
 0  &             &   \\
  \!\!\! -1 & \ddots &   \\
     & \ddots &0 \\
     &             &\!\!\! -1
  \end{matrix}
  }\end{pmatrix}

\end{array}
$$
where $J(\lambda_i)$ denotes the standard Jordan block
$$
J(\lambda_i)=\begin{pmatrix}
\lambda_i&1&&\\
&\lambda_i&\ddots&\\
&&\ddots&1\\&&&\lambda_i\end{pmatrix}.
$$
\end{theorem}

As a special case in this theorem, we consider the pair of trivial $1\times 1$ blocks $\mathcal A_i = 0$ and $\mathcal B_i = 0$.  We refer to such a situation as a {\it trivial} Kronecker block.

Notice that the choice of a canonical basis is not unique. Equivalently, one can say that the automorphism group of the pair $(\mathcal A, \mathcal B)$ is not trivial (this group has been described and studied in \cite{PumeiDiss}). However, the blocks $\mathcal A_i$ and $\mathcal B_i$ are defined uniquely up to permutation.

For the linear combination $\mathcal A + \lambda \mathcal B$ we will sometimes use the notation
 $\mathcal A_\lambda$. Besides, we will formally set $\mathcal A_\infty = \mathcal B$ having in mind that we are interested in these forms up to proportionality so that the parameter $\lambda$ of the pencil $\mathcal P = \{\mathcal A_\lambda\}$ generated by $\mathcal A$ and $\mathcal B$ belongs, in fact,  to the projective line $\mathbb CP^1$.

The rank of the pencil $\mathcal P$ is naturally defined as $\rank \mathcal P = \max_\lambda \rank \mathcal A_\lambda$.  The numbers $\lambda_i$ that appear in the Jordan blocks $\mathcal A_i$  of the Jordan--Kronecker canonical form given in Theorem \ref{JKD} are called 
{\it  characteristic numbers} of the pencil $\mathcal P$. They play the same role as ``eigenvalues'' in the case of linear operators. More precisely, $\lambda_i$ are those numbers for which the rank of $\mathcal A_\lambda$ for $\lambda=\lambda_i$ is not maximal,  i.e., $\rank \mathcal A_{\lambda_i} < \rank \mathcal P$. The case of Jordan blocks with $\lambda_i=\infty$ can always be avoided by replacing $\mathcal B$ with $\mathcal B^\prime=\mathcal B+\mu \mathcal  A$ for a suitable $\mu$. So from now on, unless otherwise stated, we shall assume that $\infty$ is not a characteristic number, so that no Jordan block with ``infinite eigenvalue'' appears.  There is a natural analog of the characteristic polynimial $\mathsf p(\lambda)$ whose roots are exactly the characteristic numbers with multiplicities. In order to define $\mathsf p(\lambda)$ in invariant terms, we consider all diagonal minors of the matrix $A+ \lambda B$ of order $\rank \mathcal P$ and take the Pfaffians for each of them. They are obviously polynomial in $\lambda$. Then $\mathsf p(\lambda)$ is the greatest common divisor of all these Pfaffians.

If $\mu \ne \lambda_i$, then we call the form $\mathcal A_\mu$ {\it regular}  (in the pencil $\mathcal P=\{\mathcal A_\lambda\}$).  The set of characteristic numbers $\lambda_i$ of the pencil $\mathcal P$ will be denoted by $\Lambda$.

The size of each Kronecker block is an odd number $2k_i - 1$, $ i=1,\dots, s$.  As we shall see below, the numbers $k_i$ in many cases have a natural algebraic interpretation and we shall call them {\it Kronecker indices} of the pencil $\mathcal P=\{ \mathcal A_\lambda\}$. Notice, by the way,  that the number of Kronecker blocks $s$ is equal to $\corank \mathcal P$.

The Jordan--Kronecker decomposition theorem  immediately implies several important facts.  First of all, we can always find a large subspace which is isotropic simultaneously for all forms from a given pencil $\mathcal P$.  Speaking more formally, we call a subspace $U\subset V$ bi-Lagrangian w.r.t. a pencil $\mathcal P$ if $U$ is isotropic for all $\mathcal A_\lambda \in \mathcal P$ and $\dim U=\frac{1}{2}(\dim V+\cork  \mathcal P )$. In other words,  $U$ is a common maximal isotropic subspace for all regular forms $A_\lambda \in \mathcal P$.

\begin{Cor}
\label{cor1}
For every pencil $\mathcal P=\{\mathcal A_\lambda\}$, there is a bi-Lagrangian subspace $U\subset V$.  \end{Cor}

\begin{proof}
The proof is evident:  as such a subspace $U$ one can take the direct sum of the subspaces related to the right lower zero blocks of  the submatrices  $\mathcal A_i$ and  $\mathcal B_i$ in the Jordan-Kronecker decomposition.
\end{proof}

In fact, this result gives an algebraic explanation of the role which compatible Poisson brackets play  in the theory of completely integrable systems: an analog of  a bi-Lagrangian subspace is just a family of functions in bi-involution. 
In particular, Corollary \ref{cor1} can be understood as an algebraic counterpart for the generalised argument shift conjecture.  By using the results of F.-J.~Turiel \cite{turiel}, \cite{tur3}  on the local classification of  compatible Poisson brackets, one can show that a local version of this conjecture holds true if we replace polynomials by local analytic functions (see also paper by P.~Olver \cite{olver}).  The problem is to ``extend'' these local functions onto the whole space  $\goth g^*$,  more precisely, to ``make them'' into polynomials. Turiel's construction uses arguments from local differential geometry which do not guarantee any kind of ``polynomiality''.

\medskip

Let us give some more straightforward corollaries of Theorem \ref{JKD} playing an important role in the theory of bi-Hamiltonian systems.

\begin{Cor}\label{cor2}\quad

\begin{enumerate}
\item The subspace $L= \sum_{\lambda\notin \Lambda} \Ker \mathcal A_\lambda$ is bi-isotropic, i.e., isotropic w.r.t. all forms $\mathcal A_\lambda \in \mathcal P$. 

\item  $L$ is contained in every bi-Lagrangian subspace $U$. Moreover, $L$ can be characterised as the intersection of all bi-Lagrangian subspaces.

\item $\dim L =  \sum_{i=1}^s k_i$, where $k_1, \dots, k_s$  are the Kronecker indices of $\mathcal P$.
\end{enumerate}
\end{Cor}

The subspace $L$ admits another useful description.  Assume that $\mathcal B$ is a regular form in $\mathcal P$ and  $v^0_1,\dots, v^0_s$ is a basis of $\Ker \mathcal B$. Consider the following recursion relations:
$$
\begin{aligned}
\mathcal A v_i^1 &= \mathcal B v_i^0,\\
\mathcal A v_i^2 &= \mathcal B v_i^1,\\
&\dots\\
\mathcal A v_i^{k} &= \mathcal B v_i^{k-1},\\
&\dots
\end{aligned}
$$   
It follows immediately from Theorem \ref{JKD} that these  relations are consistent in the following strong sense: if we have chosen some $v_i^0, \dots , v_i^{k}$ satisfying the first $k$ relations (this choice is not unique), then there is $v_i^{k+1}$ that satisfies the $(k+1)$'s relation,  so that we can continue this chain up to infinity starting from any step.

\begin{Cor}
\label{coradd}
$L$ is  the span of all vectors $v_i^k$, $i=1,\dots,s$, $k\ge 0$.
\end{Cor}

The next corollary gives a description of Kronecker pencils (i.e., with no Jordan blocks).

\begin{Cor}\label{cor3}
The following statements are equivalent:
\begin{enumerate}
\item $\mathcal P$ is of Kronecker type, i.e., the JK decomposition\footnote{Sometimes we use {\it JK} as abbreviation of {\it Jordan-Kronecker}.} of $\mathcal P$ has no Jordan blocks;
\item $\rank \mathcal A_\lambda = \rank \mathcal P$ for all $\lambda\in \bar{\mathbb C}$, i.e., $\Lambda = \emptyset$;
\item the subspace $L= \sum_{\lambda\notin \Lambda} \Ker \mathcal A_\lambda$ is bi-Lagrangian;
\item a bi-Lagrangian subspace is unique.
\end{enumerate}
\end{Cor}

The following statement allows us to compute the number of Jordan blocks (both trivial, i.e., of size $2\times 2$, and non-trivial) for each characteristic number.

\begin{Cor}\label{cor4}
Let $\mathcal P = \{ \mathcal A + \lambda \mathcal B \}$  and $\mu\ne 0$ be a characteristic number. Then 
\begin{enumerate}
\item $\corank \left( \mathcal A|_{\mathcal \Ker (\mathcal A+\mu \mathcal B)}\right) \ge \corank \mathcal P$;
\item $\corank \left( \mathcal A|_{\mathcal \Ker (\mathcal A+\mu \mathcal B)}\right) = \corank \mathcal P$ iff  the Jordan $\mu$-blocks are all trivial;
\item the number of all Jordan $\mu$-blocks is equal to
$$
\frac{1}{2} \bigl(\dim\Ker (\mathcal A+\mu \mathcal B) -
\corank\mathcal P\bigr);
$$
\item the number of non-trivial Jordan $\mu$-blocks is equal to
$$
\frac{1}{2} \left(\corank \left( \mathcal A|_{\mathcal \Ker (\mathcal A+\mu \mathcal B)}\right) -
\corank\mathcal P\right).
$$ 
\end{enumerate}
\end{Cor}

These purely algebraic and elementary results have natural analogs (in fact, direct implications) in the theory of integrable systems.  Here is a kind of dictionary that allows to translate ``linear algebra'' to ``Poisson geometry'':

\medskip

\begin{tabular}{lcl}
skew-symmetric form & $\longleftrightarrow$ & Poisson structure \\
kernel of a skew-symmetric form & $\longleftrightarrow$ & Casimir functions \\
pencil of skew-symmetric forms & $\longleftrightarrow$ & compatible Poisson brackets \\
isotropic subspace & $\longleftrightarrow$ & family of commuting functions \\
maximal isotropic subspace & $\longleftrightarrow$ & integrable system \\
bi-Lagrangian subspace & $\longleftrightarrow$ & functions in bi-involution
\end{tabular}

\medskip

Understanding this relationship allows us not only to interpret, but also to prove many important facts related to compatible Poisson structures and bi-Hamiltonian systems.
For example, 
the argument shift method (part 1 of Theorem \ref{MFth}) is just a reformulation of item 1 of Corollary \ref{cor2} in terms of the compatible Poisson brackets \eqref{LiePoisson} and \eqref{abracket} on the dual space $\goth g^*$.  
The passage from the classical shifts $f(x+\lambda a)$ to the canonical algebra of polynomial shifts $\mathcal F_a$ is equivalent to the interpretation of the subspace  $L$  given by Corollary \ref{coradd}   (here  the family (algebra) of shifts $\mathcal F_a$ itself corresponds to $L$). 
The  reformulation of the generalised argument shift conjecture given at the end of Section 4 becomes immediately clear, if we compare it with item 2 of Corollary \ref{cor2}. 

In fact, the main idea of this paper is just to use this relationship in a systematic way for compatible  Poisson brackets \eqref{LiePoisson} and \eqref{abracket} on the dual space $\goth g^*$ in order to get some information about the Lie algebra $\goth g$ itself and its coadjoint representation.

We are not going to give detailed proofs of the results presented below.  Instead, we will give a reference to one of the algebraic results discussed above, from which the desired fact immediately follows.  For a reader who is not familiar with this ``linear algebra $\longleftrightarrow$ Poisson geometry'', we should, perhaps, explain from the very beginning how deep this relationship is.
A Poisson structure, as an object of differential geometry,  can be considered up to a certain order of  approximation.  The above relationship is just of zero order.  But even this leads to non-trivial results, as behind ``linear algebra''  there is always a ``compatibility condition''\footnote{This condition is of the first order  (and non-linear!) and says that $\mathcal A+ \lambda\mathcal B$ is Poisson  for each $\lambda$.}  which is highly non-trivial and is responsible for many things that cannot be even seen on the level of ``linear algebra''.   For example,  one can naturally ask the following question: how the Jordan--Kronecker decomposition  could help in the theory of Lie algebras, if it does not know anything about the Jacobi identity?  The answer is very simple:  Jacobi identity is hidden in the compatibility condition for brackets  \eqref{LiePoisson} and \eqref{abracket}.  Thus, the point is that the Jordan--Kronecker decomposition indeed contains a lot of useful information,  but to get any non-trivial conclusion  from it, we need something extra, and this ``extra'' is basically hidden in the compatibility condition.

\section{Definiton of Jordan--Kronecker invariants}

Let $\mathfrak g$ be a Lie algebra and $\mathfrak g^*$ be its dual space. Consider $x,a\in\goth g^*$ and the corresponding skew-symmetric forms $\mathcal A_x=\Bigl( \sum c_{ij}^k x_k \Bigr)$ and $\mathcal A_a=\Bigl( \sum c_{ij}^k a_k \Bigr)$. The Jordan--Kronecker decomposition of the pencil $\{\mathcal A_x + \lambda \mathcal A_a\}$ essentially depends on the choice of $x$ and $a$.   However,  for almost all pairs $(x,a)$, the algebraic type of this pencil is the same.

We will say that $(x,a)\in \mathfrak g^*\times \mathfrak g^*$ is a {\it generic pair}  if the algebraic type of the Jordan--Kronecker decomposition  for $\{\mathcal A_x + \lambda \mathcal A_a\}$ remains unchanged under a small variation of both $x$ and $a$. The pencil $\{\mathcal A_{x} + \lambda \mathcal A_{a}\}$  in this case is called {\it generic} too. Clearly, the set of all generic pairs $(x,a)$ is Zariski open non-empty subset of $\mathfrak g^*\times \mathfrak g^*$.

\begin{definition}\rm{
The algebraic type of  the Jordan--Kronecker canonical form for a generic pencil $\mathcal A_x + \lambda \mathcal A_a$  is called the {\it Jordan--Kronecker invariant} of $\mathfrak g$.}
\end{definition}

Here by the algebraic type of a JK canonical form we mean the number and sizes of Kronecker  and Jordan (of course, separately for each characteristic number) blocks. 

In particular,  we will say that a Lie algebra $\goth g$  is of
\begin{itemize}
\item Kronecker type,
\item Jordan (symplectic) type,
\item  mixed type,
\end{itemize}
if the Jordan--Kronecker decomposition for the generic pencil $\mathcal A_x + \lambda \mathcal A_a$ consists of 
\begin{itemize}
\item only Kronecker blocks,
\item only Jordan blocks,
\item both of Jordan and Kronecker blocks
\end{itemize}
respectively.

Following the same idea we give two more definitions.

\begin{definition}{\rm
The Kronecker indices of a generic pencil  $\mathcal A_x + \lambda \mathcal A_a$ are called {\it Kronecker indices} of $\goth g$.}
\end{definition}

\begin{definition}\rm{ The characteristic numbers of a generic pencil
 $\mathcal A_x + \lambda \mathcal A_a$ are called {\it characteristic numbers $\lambda_i$} of $\mathfrak g$.}
\end{definition}

Notice that in a neighbourhood of a generic pair $(x,a)\in \goth g^*\times\goth g^*$, these characteristic numbers are analytic functions of $x$ and $a$:
$$
\lambda_i = \lambda_i(x,a).
$$

\section{Basic properties of JK invariants}

The next two theorems easily follow from the definition of JK invariants and give characterisation of Lie algebras of Kronecker and Jordan types respectively.

\begin{theorem}
\label{kroncase}
The following properties of a Lie algebra $\goth g$ are equivalent:
\begin{enumerate}
\item $\goth g$ is of Kronecker type, i.e. the Jordan--Kronecker decomposition of a generic pencil  $\mathcal A_x + \lambda \mathcal A_a$ consists only of Kronecker blocks,
\item $\codim \Sing  \ge 2$,  where
$$
\Sing  =\{  y\in \goth g^*~|~   \dim\Ann y > \ind\goth g   \} \subset \goth g^*
$$
is the singular set of $\goth g^*$,
\item  the algebra of shifts $\mathcal F_a$ is complete.
\end{enumerate}

\end{theorem}

\begin{proof}
This theorem is, in fact,  the main result of \cite{Bols1} .
We give a sketch of  proof (see details in \cite{Bols1} and, in a more general case, \cite{bz}).   A generic pencil $\mathcal A_x + \lambda \mathcal A_a$ is Kronecker,  if and only iff the rank of $\mathcal A_x + \lambda \mathcal A_a=\mathcal A_{x+\lambda a}$  is maximal for all $\lambda$ (Corollary \ref{cor3}), i.e., a generic line $x+\lambda a$ does not intersect the singular set $\Sing $.  This is obviously equivalent to the condition $\codim \Sing  \ge 2$.  The equivalence of 1 and 3 follows directly from Corollary \ref{cor3}  (see items 1 and 3) if we take into account the fact that the differentials of shifts $f\in\mathcal F_a$ at the point $x$ generate the subspace $L=\sum_{\lambda\notin\Lambda} \Ker (\mathcal A_{x}+\lambda \mathcal A_{a})$ (Corollary \ref{coradd}).
\end{proof}

Notice that for Lie algebras of Kronecker type, the generalised argument shift conjecture holds true automatically as the family of shifts $\mathcal F_a$  itself is complete and in bi-involution.  Examples of such Lie algebras include, first of all, semisimple Lie algebras \cite{MischFom} and semiderect sums $\goth g +_\rho V$, where $\goth g$ is simple, $V$ is Abelian  and $\rho: \goth g \to \mathrm{gl}(V)$ is irreducible \cite{bolsActa}, \cite{priwitzer}, \cite{KnopLittl} (see Section \ref{examples}).

\medskip 

The next theorem is obvious and can be viewed as an interpretation of the notion of a Frobenius Lie algebra (\cite{ElashFrob},  \cite{Ooms})  in terms of Jordan--Kronecker invariants.

\begin{theorem}
\label{jordcase}
The following properties of a Lie algebra $\goth g$ are equivalent:
\begin{enumerate}
\item $\goth g$ is of Jordan type, i.e., the Jordan-Kronecker decomposition for a generic pencil  $\mathcal A_x + \lambda \mathcal A_a$ consists only of Jordan blocks.
\item a generic form $\mathcal A_x$ is non-degenerate, i.e.,
$\ind \goth g=0$  and $\goth g$ is Frobenius,
\item the algebra of shifts $\mathcal F_a$ is trivial, i.e., $\mathcal F_a=\mathbb C$.
\end{enumerate}
\end{theorem}

\section{Kronecker blocks and Kronecker indices}

Here we focus on  Kronecker blocks and discuss  some elementary results to illustrate a relationship between JK invariants and properties of a Lie algebra~$\goth  g$.

\begin{proposition}
Let $\mathcal P = \{ \mathcal A_{x+\lambda a}\}$ be a generic pencil,  $x,a\in\goth g^*$. Then:
\begin{enumerate}
\item  the number of Kronecker blocks in the JK decomposition for $\mathcal P$ equals to the index of $\goth g$;

\item the number of trivial Kronecker blocks is greater or equal to the dimension of the centre of $\goth g$;

\item  the number of algebraically independent functions in the algebra of shifts $\mathcal F_a$ equals $\sum_{i=1}^s k_i$,  where
$k_1,\dots, k_s$  are Kronecker indices of $\goth g$,  $s=\ind\goth g$.
\end{enumerate}
\end{proposition}

Items 1 and 2 are obvious. The third statement follows from Corollary \ref{cor2}, part 3.

\medskip

It is interesting to notice that Kronecker indices give a very simple and natural estimate  
for the degrees of polynomial coadjoint invariants. This result has been recently obtained by A.\,Vorontsov.

\begin{theorem} [A.\,Vorontsov \cite{Vorontsov}]\label{Vorontsov}
Let $f_1(x), f_2(x), \dots, f_s(x)\in P(\goth g)$ be algebraically independent polynomial coadjoint invariants,  $s=\ind\goth g$,  and  $m_1\le m_2\le \dots \le m_s$  be their degrees,  $m_i=\mathrm{deg}\, f_i$. Then the following estimate holds 
\begin{equation}
\label{estim}
m_i \ge k_i,
\end{equation}
where $k_1\le k_2 \le\dots \le k_s$ are Kronecker indices of $\goth g$.
\end{theorem}

This theorem is related to the case when $\goth g$ admits a ``complete set'' of polynomial $\Ad^*$-invariants, i.e., the number of algebraically independent invariants is equal to the index of $\goth g$.  
However, in general  coadjoint invariants are not necessarily polynomial or even rational\,\footnote{The rationality of invariants is guaranteed by Rosenlicht's Theorem for algebraic Lie algebras.  The polynomiality follows, in particular, from the condition $[\goth g, \goth g]=\goth g$.  A Lie algebra satisfying this condition is called {\it perfect}.}.  But even in this more general case a similar estimate holds true.  We only need to replace the degree $m_i$ by another characteristic of a (local) analytic function $f$.  Namely,  consider the Taylor expansion of $f$ at a generic point $a\in \goth g^*$:
$$
f(a + \lambda x) = f_0 + \lambda f_1(x) + \lambda^2 f_2(x) + \lambda^3 f_3(x) + \dots
$$
where $f_k$ is a homogeneous polynomial in $x$ of degree $k$.  Denote by $m(f)$ the number of algebraically independent polynomials among $f_i$'s.   It is clear that  if $f$ itself is a polynomial, then $m(f) \le \deg f$.  

Thus, if $f_1,\dots, f_s$  are independent (local) analytic $\Ad^*$-invariants, 
$s=\ind\goth g$,  and  $m(f_1)\le m(f_2)\le \dots \le m(f_s)$, then we still have the same estimate
$$
m(f_i) \ge k_i, \quad
i=1,\dots, s=\ind \goth g.
$$

This observation can sometimes be used to compute Kronecker indices for Lie algebras.  For example,  for a semisimple Lie algebra $\goth g$, the algebra of  polynomial invariants admits a natural basis 
$f_1, \dots, f_s$  and in this case \eqref{estim}  becomes the exact equality (A.\,Panasyuk \cite{Pan1}):
$$
m_i =  k_i.
$$
The numbers $e_i=m_i-1$  are known as {\it exponents} of a semisimple Lie algebra $\goth g$.   Thus, the Kronecker indices of  $\goth g$ can naturally be related to its exponents.  This result is based on the following general observation \cite{Vorontsov}.

\begin{proposition}
\label{pv}
Let $\goth g$ be of Kronecker type and $f_1,\dots, f_s$,  $s=\ind g$, be algebraically independent $\Ad^*$-invariant polynomials.  If 
\begin{equation}
\sum_{i=1}^s \deg f_i = \frac{1}{2}(\dim \goth g + \ind\goth g),
\label{eqs}
\end{equation}
 then 
$$\deg f_i=k_i ,$$
where $k_i$ are the Kronecker indices of $\goth g$.  In particular, the degrees of algebraically independent $\Ad^*$-invariants $f_1,\dots, f_s$  satisfying \eqref{eqs} are uniquely defined.
\end{proposition}

It is worth noticing that  $\Ad^*$-invariants $f_1,\dots, f_s$  satisfying \eqref{eqs}  possess some other interesting properties (see, for example,  \cite{PY}, \cite{panyushev2}, \cite{OdessRub}).  Here is, for example,  a generalisation of  a well-known result due to B.\,Kostant  (Theorem
9, p. 382 in \cite{K2}) to the non-semisimple case.  

\begin{proposition}
\label{genKostant}
Let $\goth g$ be of Kronecker type and $f_1,\dots, f_s$,  $s=\ind g$, be algebraically independent $\Ad^*$-invariant polynomials satisfying \eqref{eqs}. Then the differentials of $f_1,\dots, f_s$ at a point $x\in\goth g^*$  generate
$\Ann x$ if and only if $x$ is regular, i.e., $x\notin \Sing$.
\end{proposition}

\begin{proof} The proof can be found in \cite{panyushev2}, but here we give another version based on the concept of JK invariants.

Let $x\notin \Sing$.  Since $\goth g$ is of Kronecker type, i.e.,   $\codim\Sing \ge 2$,  we can find a regular element $a\in \goth g^*$ such that the line $x+\lambda a$ does not intersect $\Sing$, so that the pencil $\mathcal A_x + \lambda \mathcal A_a$ is Kronecker.

Consider the generators $f_i^k$ of the algebra of shifts $\mathcal F_a$ obtained from $f_1,\dots, f_s$ as explained in Section 3 (see \eqref{generators}),   and the subspace $L\subset\goth g$ generated by their differentials $df_i^k(x)$. According to Corollary \ref{cor3},  $L$ is bi-Lagrangian, i.e.,
$\dim L = \frac{1}{2}(\dim \goth g + \ind\goth g)$.  But due to \eqref{eqs} this number coincides with the number of generators $f_i^k$.  Hence the differentials $df_i^k(x)$ are linearly independent. The same is, therefore, true for $df_1(x),\dots, df_s(x)$, as the invariants $f_1,\dots, f_s$ themselves belong to the set of generators $\{f_i^k\}$, namely $f_i=f_i^{\deg f_i}$.

Thus, $df_i(x)$ are linearly independent at $x$ and therefore generate $\Ann x$, as required. The converse is obvious:  $x\in\Sing$  means that $\dim\Ann x > s=\ind \goth g$ and $df_1,\dots, df_s$ cannot generate this subalgebra. \end{proof}

\section{Singular set and characteristic numbers}

The singular set $\Sing\subset \goth g^*$ plays a very important role in our construction. Here we briefly discuss some of its elementary properties.

As a subset of $\goth g^*$,  the singular set $\Sing$   
is an algebraic subvariety given by the system of homogeneous polynomial equations of the form:
\begin{equation}
\mathrm{Pf}\, C_{i_1i_2\dots i_{2k}}=0,  \quad 1\le i_1 < i_2 < \dots < i_{2k} \le \dim \goth g
\label{singeqs}
\end{equation}
where $\mathrm{Pf}$ denotes the Pffafian, and $C_{i_1i_2\dots i_{2k}}$ is the diagonal minor of the skew-symmetric matrix
 $\mathcal A_x = (c_{ij}^kx_k)$, related to the rows and columns with numbers $i_1,i_2, \dots , i_{2k}$, $2k=\dim\goth g - \ind\goth g$.  The case of Abelian Lie algebra should,  perhaps, be considered as an exception:  in this case $\Sing = \emptyset$.  Otherwise, $\Sing$ is not empty and contains at least the zero element.

$\Sing $ may consist of several irreducible components which, in general, may have different dimensions. One of the simplest examples is  the direct sum $\goth g=\goth g_1\oplus \goth g_2$, where   the singular sets $\Sing_i\subset \goth g^*_i$ ($i=1,2$) have different codimensions. Then the singular set for $\goth g$ is $\Sing = (\Sing_1\times\goth g_2^*)\cup (\Sing_2\times\goth g_1^*)$, i.e. consists of two components with different dimensions. 

The codimension of $\Sing$ can be arbitrarily large.  As an example, consider the semidirect sum of one-dimensional Lie algebra $\goth h$ and an $n$-dimensional vector space $V$, where a generator $h\in \goth h$ acts on $V$ as a regular semisimple operator. Then it is easy to check that $\Sing \subset (\goth h + V)^*$ is one-dimensional (a line), i.e., $\codim \Sing =   n$. 

If $\goth g$ is a semisimple Lie algebra, then  $\codim \Sing =3$.

The structure of the singular set $\Sing $ plays a very important role in the case when it has codimension $1$.  As we pointed out above,  it is quite possible that $\Sing $ at the same time possesses irreducible components of higher codimension. In such a case it is convenient to distinguish in  $\Sing $  the subset $\Sing _0$ that is the union of all components of codimension 1.  In other words, we simply remove from $\Sing $ all ``low-dimensional'' components.   Then $\Sing _0$ is an algebraic variety defined by one homogeneous polynomial equation:
$$
\Sing _0 = \{ f(x) = 0\}
$$

Such a polynomial $f$ is easy to describe in terms of the structure tensor $c_{ij}^k$.  Indeed,  $\Sing$ is defined by the system of equations \eqref{singeqs}.  Since $\Sing_0 \subset \Sing$, then  $f(x)=0$ implies vanishing all Pfaffians $\mathrm{Pf}\, C_{i_1i_2\dots i_{2k}}$.  Hence, as $f(x)$ we can simply take the greatest common divisor of all these Pfaffians. 
\begin{equation}
\mathsf{f} (x) = \mathrm{g.c.d.} \Bigl(\mathrm{Pf}\, C_{i_1i_2\dots i_{2k}}, \quad 1\le i_1 < i_2 < \dots < i_{2k} \le \dim \goth g\Bigr)
\label{f}
\end{equation}

This polynomial $\mathsf{f}(x)$ is not necessarily irreducible and, in general, may be decomposed into product of (irreducible) components:
$$
\mathsf{f}(x) = \underbrace{f_1(x)\cdot \ldots \cdot f_1(x)}_{s_1 \  \mathrm{ times}}\cdot \ldots \cdot 
 \underbrace{f_k(x)\cdot \ldots \cdot f_k(x)}_{s_k \  \mathrm{ times}}.
$$

Notice that $\mathsf{f}$  (as well as each of irreducible factors $f_i$) is a coadjoint semi-invariant, i.e.,
satisfies $\mathsf{f}(\Ad^*_g x )=\chi(g)\cdot \mathsf{f}(x)$, where $\chi$ is a certain character of the corresponding Lie group.

Along with the polynomial $\mathsf{f}$, we will consider its {\it reduced version}:
\begin{equation}
\mathsf{f}_{\mathrm{red}}(x) = f_1(x)\cdot \ldots \cdot f_k(x),
\label{fred}
\end{equation}
i.e., each irreducible components appears with multiplicity one.  Clearly,  $\mathsf{f}_{\mathrm{red}}(x)=0$ still defines the codimension one singular set $\Sing_0$.

The set $\Sing_0$ and polynomials $\mathsf{f}$, $\mathsf{f}_{\mathrm{red}}$ are closely related to the characteristic numbers of the Lie algebra $\goth g$. First of all, Theorem \ref{kroncase} immediately implies

\begin{proposition}
Characteristic numbers of $\goth g$ exist if and only if  $\codim \Sing  =1$.
\end{proposition}

Let $\mathcal P = \{\mathcal A_{x+\lambda a}\}$ be a generic pencil, $x,a\in\goth g^*$. 
The characteristic numbers $\lambda_i=\lambda_i(x,a)$ can be characterised by the simple algebraic condition that  $x+\lambda_i a \in \Sing$.  Since the pair $(x,a)$ is generic, $\Sing$ can be replaced by its codimension one part $\Sing_0$. In other words, the characteristic numbers are exactly the roots of the polynomial  $\mathsf{p}(\lambda)=\mathsf{f}(x+ \lambda a)$ (or $\mathsf{p}_{\mathrm{red}}(\lambda)=\mathsf{f}_{\mathrm{red}}(x+ \lambda a)$). According to decomposition \eqref{singeqs}  (or \eqref{fred}),  the  characteristic numbers can be partitioned into $k$ groups $\Lambda_1, \dots, \Lambda_k$ each of which naturally corresponds to one of these irreducible polynomials $f_1(x),\dots, f_k(k)$.
Hence we immediately obtain

\begin{proposition}\label{reducible}\quad

\begin{enumerate}
\item The number of distinct characteristic numbers $\lambda_i$  of  $\goth g$ equals  the degree of  $\mathsf{f}_{\mathrm{red}}(x)$.  Similarly, the degree of $\mathsf{f}(x)$ is the number of characteristic numbers with multiplicities.

\item More precisely,  the number of characteristic numbers in each group $\Lambda_i$ is equal to the degree of  $f_i$.  The multiplicity\footnote{By the multiplicity of a characteristic number $\lambda_i$ we understand the sum of sizes of the corresponding Jordan blocks $J(\lambda_i)$, see Theorem \ref{JKD}.  We want to emphasise that due to skew symmetry each block in the JK decomposition consists of the pair of  $J(\lambda_i)$ blocks.  To compute the multiplicity we take into account only one of them.} of  a characteristic number from $\Lambda_i$ is equal to the multiplicity $s_i$ of $f_i$ in the decomposition   \eqref{singeqs}. In particular, all characteristic numbers within a group $\Lambda_i$ have the same multiplicity.

\item If some of the characteristic numbers have different multiplicities, then the set (variety) of singular elements $\Sing_0 $ is reducible.

\end{enumerate}

\end{proposition}

Recall that speaking of characteristic numbers $\lambda_i$ of $\goth g$ we consider them as locally analytic functions $\lambda_i(x,a)$ defined in a neighbourhood of a generic pair $(x,a)\in\goth g^*\times\goth g^*$.  However, the partition of the characteristic numbers into groups $\Lambda_i$ is global, whereas the characteristic numbers within a certain group $\Lambda_i$ are defined only locally. For applications, we need, of course,  globally defined invariants of the pencil 
$\mathcal A_{x+\lambda a}$.  They can be easily constructed by means of Vi\`ete's theorem.

\begin{proposition}
The symmetric polynomials of characteristic numbers  are rational functions of $x$ and $a$.
Moreover, if $a\in \mathfrak g^*$ is fixed, then
they are polynomial in $x$.
\end{proposition}

In this statement,  we can consider all distinct characteristic numbers, or all characteristic numbers with multiplicities, or all characteristic numbers from a certain group $\Lambda_i$. The conclusion of this proposition holds true in each of these cases.

From the viewpoint of the generalised argument shift conjecture,  the following statement is very important.

\begin{proposition}
Consider the polynomial $\mathsf{f}_{\mathrm{red}}(x)$ defining the codimension one singular set $\Sing_0$ and given by \eqref{fred}.  Take the Taylor expansion of $\mathsf{f}_{\mathrm{red}}$ at the point $a\in\goth g^*$:
$$
\mathsf {f}_{\mathrm{red}} (a + \lambda x) = g_0 + \lambda g_1(x) + \lambda^2 g_2(x) + \dots + \lambda^m g_m(x).
$$
Then the homogeneous polynomials $g_1(x), \dots, g_m(x)$  are in bi-involution w.r.t. the brackers \eqref{LiePoisson} and \eqref{abracket}.
\end{proposition}

Clearly, the polynomials $g_1, \dots, g_m$ up to a certain constant  (that depends on $a$) are exactly the symmetric polynomials of characteristic numbers. So this proposition is just a particular case of a well-known statement from the theory of bi-Hamiltonian systems:   for  any pencil of compatible Poisson structures $\mathcal A +\lambda\mathcal B$, the characteristic numbers of this pencil are in bi-involution.

On the other hand, this proposition can be considered as a particular case of the ``shift of semi-invariants'' method suggested by A.A.\,Arkhangelskii \cite{Arkh} and then developed by V.V.\,Trofimov \cite{trofimov}.

\medskip

Finally we consider the case when a Lie algebra $\goth g$ is Frobenius, i.e. its index is zero. Then $\Sing $ is defined by one single polynomial, namely: $\mathsf{f}(x)=\mathrm{Pf} (\mathcal A_x) = \sqrt{\det \left(c_{ij}^k x_k\right)}$. Assume that this polynomial is either irreducible, or has no multiple components in its decomposition \eqref{singeqs} into irreducibles polynomials, i.e., all $s_i$ equal 1.  This is equivalent to the fact that its degree $\deg \mathsf{f}$ coincides with the (geometric) degree of the singular set $\Sing$ which we can understand as the number of distinct intersection points of $\Sing$ with a generic line $x+\lambda a$. Such a situation seems to be quite typical. Under this assumption we have the following
 
 \begin{theorem}\label{thfrob}
Let $\goth g$ be a Frobenius Lie algebra, and the (geometric) degree of  $\Sing\subset \goth g^*$ is equal to $k=\frac12\dim\goth g$.
 Then a generic pencil $\mathcal A_x + \lambda \mathcal A_a$ is diagonalisable (i.e. has no Jordan blocks of size greater than $2\times 2$), all characteristic numbers are distinct, and the coefficients of the ``characteristic polynomial''  $\mathsf{p}(\lambda)=\mathrm{Pf}\, \mathcal A_{x + \lambda a}$ form a complete family of polynomials in bi-involution.
\end{theorem}

\begin{proof} The diagonalisability of $\mathcal A_x + \lambda \mathcal A_a$ is obvious as all characteristic numbers are distinct.
The second statement of the theorem contains one non-trivial ingredient: from the existence of 
 $k$ {\it distinct} characteristic numbers\footnote{To prove the theorem we can obviously pass from coefficients of $\mathsf{f}(x+\lambda a)$ to its roots, i.e., to the characteristic numbers.} we can immediately conclude that they are {\it functionally independent}  (by the way, it is for this reason that we need the Jacobi identity). The explanations of this ``miracle'' comes from the theory of bi-hamiltonian systems and compatible Poisson brackets.  If we consider the so-called recursion operator $R=\mathcal A_x\mathcal A_a^{-1}$, then the compatibility condition for the Poisson structures $\mathcal A_x$ and $\mathcal A_a$ immediately implies vanishing the Nijenhuis tensor for $R$. It is a very well-known fact from local differential geometry that non-constant eigenvalues of such operators have to be functionally independent. The point is that $R$  (with zero Nijenhuis tensor) can locally be reduced to a block-diagonal form where each block possesses exactly one eigenvalue and, moreover, this eigenvalue depends only of the coordinates related to the block\footnote{Alternatively, one can use the normal form theorem for non-degenerate compatible Poisson structures by F.-J.Turiel \cite{turiel} from which the desired result immediately follows.}. Thus, the purely algebraic fact (algebraic independence of the coefficients of  $\mathsf{p}(\lambda)=\mathrm{Pf}\, \mathcal A_{x + \lambda a} $) which would  probably be difficult to prove by algebraic means, turns out to be almost obvious from the viewpoint of bi-Poisson geometry. \end{proof}

Notice that if  the degree of $\Sing$ is smaller than $\frac12\dim\goth g$, then in the case of a Frobenius Lie algebra $\goth g$ we can still assert that the coefficients of the reduced polynomial $\mathsf{p}_{\mathrm{red}}(\lambda)=\mathsf{f}_{\mathrm{red}}(x+\lambda a)$  are functionally independent, i.e., in any case we obtain $k$ functions in bi-involution, where $k$ is the geometric degree of $\Sing$.
It is not quite clear if this statement still holds if $\goth g$ is not Frobenius  (i.e. if some Kronecker blocks exist). The answer is apparently  negative.

\section{Index of the annihilator and the Elashvili conjecture from the viewpoint of JK invariants}

In this section, instead of a generic pair $(x,a)\in \goth g^*\times\goth g^*$ we consider the situation when $a\in \goth g^*$ is singular and fixed, whereas $x\in \goth g^*$ is still generic.

Let $\Ann a = \{ \xi \in \goth g~|~ \ad^*_\xi a = 0\}$ be the stationary subalgebra of $a\in \goth g^*$ with respect to the coadjoint representation.
The following estimate is well-known (see, for example, \cite{ArnGiv}):
\begin{equation}
\ind \Ann a \ge \ind \goth g.
\label{indexest}
\end{equation}

The Elashvili conjecture\footnote{This conjecture has its origin in the theory of integrable systems on Lie algebras. Namely, in \cite{Bols1} it was proved that the condition $\ind \Ann a = \ind \goth g$ is equivalent to the completeness of the family of shifts on the singular coadjoint orbit $O(a)$. This equality was checked for all singular elements of $\mathrm{sl}(n)$ and it was conjectured that it is still true for arbitrary (or at least for classical) semisimple Lie algebras.} states that if $\goth g$ is semisimple then for any $a\in \goth g^*=\goth g$ we have the equality
\begin{equation}
\ind \Ann a = \ind \goth g.
\label{indexconj}
\end{equation}
This conjecture has been recently proved by J--Y.\,Charbonnel and A.\,Moreau \cite{charb}, see also discussion in \cite{degraaf, panyushev, yak}.

Here is the reformulation of \eqref{indexconj}  in terms of Jordan-Kronecker decomposition:

\begin{proposition}
\label{proindex}
Let $a\in \goth g^*$ be fixed and $x\in \goth g^*$ be generic in the sense that the type of the  Jordan--Kronecker decomposition of the pencil $\mathcal A_x+\lambda \mathcal A_a$ remains unchanged under small perturbation of $x$. Then
$$
\ind \Ann a = \ind \goth g
$$
if and only if the Jordan--Kronecker decomposition of $\mathcal A_x+\lambda \mathcal A_a$ does not contain non-trivial Jordan blocks, i.e.,  the Jordan part is diagonalisable.
Otherwise, i.e. if there are non-trivial Jordan blocks,  we have strong inequality:
$$
\ind \Ann a > \ind \goth g
$$
\end{proposition}

\begin{proof}
This result is just a reformulation of items 1 and 2 of Corollary \ref{cor4}  for  this special pencil $\mathcal P=\{\mathcal A_x+\lambda \mathcal A_a\}$. 
Indeed \eqref{indexest} is a particular case of item 1   (when $\mu=\infty$).
For our pencil $\corank \mathcal P=\ind \goth g$,
$\Ker \mathcal B = \Ann a$ and 
$\mathcal A|_{\Ker \mathcal B}$ is just the skew-symmetric form on $\Ann a$ related to the element $\pi(x) \in (\Ann a)^*$ where $\pi: \goth g^* \to (\Ann a)^*$ is the natural projection. In particular,   if $x$ is generic, then we have  $\corank \left( \mathcal A|_{\Ker \mathcal B} \right) = \dim \Ker  \left( \mathcal A|_{\Ker \mathcal B} \right) = \ind \Ann a$.  Item 2 of Corollary \ref{cor4} is then equivalent to the desired conclusion. \end{proof}

It would be interesting to understand if this observation could lead to another proof of the Elashvili conjecture  or/and to its generalisation to another classes of Lie algebras (not necessarily semisimple).

An example of a strict inequality in \eqref{indexconj} is given in the next section where we discuss the Lie algebra $\mathrm{gl}(n)+\R^{n^2}$.

The above discussion can be helpful to answer the following question.
 Let $\lambda=\lambda(x,a)$ be a characteristic number of a Lie algebra $\goth g$, i.e. $x+\lambda a\in \Sing $. What can we say about the number and sizes of the corresponding Jordan $\lambda$-blocks?
 
\begin{proposition}\label{predlast}
\quad

{\rm 1)} The number of Jordan $\lambda$-blocks is equal to $\frac{1}{2}(\dim \Ann (x+\lambda a) - \ind\goth g)$.

{\rm 2)} The number of non-trivial $\lambda$-blocks  (i.e. of size greater than $2\times 2$) is equal to
$\frac{1}{2}(\ind \Ann (x+\lambda a) - \ind\goth g)$.
\end{proposition}

\begin{proof}
See items 3 and 4 of Corollary \ref{cor4}.
\end{proof}

This statement is useful for computing JK invariants  (see next Section).

\section{Examples}\label{examples}

There are only  a few examples  where JK invariants have been explicitly described.   In this section we discuss some types of Lie algebras for which this can be done.

\subsection{Semisimple case}

As was already mentioned,  a semisimple Lie algebra $\goth g$ is of  Kronecker type and its Kronecker indices $k_1,\dots, k_s$, $s=\ind \goth g=\rank \goth g$ coincide with   the degrees  of basis invariant polynomials of $\goth g$. Equivalently, $k_i = e_i + 1$, where $e_1,\dots, e_s$ are exponents of $\goth g$.

So for simple Lie algebras, the Kronecker indices are as follows:
\begin{itemize}
\item  $A_n$: \quad $2,3,4,\dots, n+1$;
\item $B_n$: \quad $2,4,6,\dots, 2n$;
\item $C_n$: \quad $2,4,6,\dots, 2n$;
\item $D_n$: \quad $2,4,6,\dots, 2n-2$ and $n$;
\item $G_2$: \quad $     2,6$;
\item $F_4$: \quad $      2, 6, 8,  12$;
\item $E_6$: \quad $     2, 5, 6, 8, 9, 12$;
\item $E_7$: \quad $      2, 6, 8, 10, 12, 14, 18$; 
\item $E_8$: \quad $      2, 8, 12, 14, 18, 20, 24, 30$.
\end{itemize}

\subsection{Semidirect sums}

As an example,  consider first the Lie algebra  $e(n)=\mathrm{so}(n) + \R^n$  of the  group of affine orthogonal transformations.
We know that the algebra $\mathcal F_a$ of shifts for this Lie algebra is complete \cite{bolsActa}.  This means that $e(n)$ is  of Kronecker type.
To determine the Kronecker indices $k_i$ of $e(n)$,  we may apply Proposition \ref{pv}.  It is well know that the basis coadjoint invariants of $e(n)$  have the same degrees $m_i$ as those of $\mathrm{so}(n+1)$  (in fact, there is a natural relationship between  the invariants of $\mathrm{so}(n+1)$ and $e(n)$ based on the fact that $e(n)$ can be obtained from $\mathrm{so}(n+1)$ by the so-called $\mathbb Z_2$-contraction).  Since in this case we have the exact equality
$$
 \sum m_i = \frac12\bigl(\dim e(n) + \ind e(n)\bigr) = \frac12\bigl(\dim \mathrm{so}(n+1) + \ind \mathrm{so}(n+1)\bigr),
$$
then the Kronecker indices of $e(n)$ are exactly $k_i=m_i$. In other words, the JK invariants of the Lie algebras $e(n)$ and $\mathrm{so}(n+1)$ coincide.

It is natural to conjecture that  a similar statement holds true in the following more general situation.  Let $\goth g$ be a semisimple Lie algebra with  $\Z_2$-grading:
$$
\goth g = \goth k + \goth p, \qquad  [\goth k,\goth k]\subset \goth k, \quad [\goth k,\goth p]\subset \goth p, \quad
[\goth p,\goth p]\subset \goth p,
$$

Then we can construct a new Lie algebra $\tilde{\goth g}$ that coincides with $\goth g$ as vector space and but  $\goth p$  becomes a commutative ideal  (whereas the commutation relations within $\goth k$ and between $\goth k$ and $\goth p$ remain the same as in $\goth g$).  In such a situation, one says that $\tilde{\goth g}$ is obtained from $\goth g$ by $\mathbb Z_2$-contraction.   
In the above example, $e(n)$ and $\mathrm{so}(n+1)$ are related exactly in this way.  Our conjecture is that the JK invariants of  $\tilde{\goth g}$ coincide with those of the semisimple Lie algebra~$\goth g$. In other words,  JK invariants survive under $\Z_2$-contraction.

Another interesting example of a semidirect sum is the Lie algebra 
$\goth g =\mathrm{sl}(n) + \R^n$.  In this case,  there is only one co-adjoint invariant polynomial. Its degree $m$ is exactly equal to 
$\frac12(\dim \goth g  + \ind \goth g) = \frac12(n^2  +n)$.  We also know that  $\mathcal F_a$ is complete \cite{bolsActa}.  Hence we conclude that 
$\goth g =\mathrm{sl}(n) + \R^n$ is a Lie algebra of Kronecker type with one Kronecker block whose size, therefore, equals to $\dim \goth g$.
Notice, however, that for this conclusion the information about the degree of the co-adjoint invariant is not essential.

More generally,  consider the semidirect sum $\goth g +_\phi V$, where $\goth g$ is simple and $\phi: \goth g \to \mathrm{End} (V)$ is irreducible.   Such Lie algebras are all of Kronecker type.  This fact amounts to the condition $\codim \Sing \ge 2$ which is not obvious at all and follows from three papers \cite{KnopLittl}, \cite{bolsActa}, \cite{priwitzer}.  For some of these Lie algebras, the Kronecker indices can be found by using Proposition \ref{pv}, but in general the question is open.

\subsection{Lie algebra of upper triangular matrices}

Let $\goth t_n$ be the Lie algebra of upper triangular $n\times n$ matrices.  The description of  Jordan-Kronecker invariants for $\goth t_n$ easily follows from a very interesting paper \cite{Arkh} by A.Arkhangelskii.  The main result of \cite{Arkh} is a proof of the generalised argument shift conjecture for $\goth t_n$  (the bracket $\{~,~\}_a$ was not discussed in  \cite{Arkh}, but  the complete family of commuting polynomials constructed by  A.Arkhangelskii  is, in fact,  in bi-involution).

If $n$ is even, then $\goth t_n$ is of mixed type, i.e., the Jordan-Kronecker  decomposition of a generic pencil $\{\mathcal A_{x+\lambda a}\}$ contains both Kronecker and Jordan blocks.  
The Kronecker indices are closely related to the coadjoint invariants of $\goth t_n$ explicitly  described in \cite{Arkh}. These invariants are rational functions $f_k = \frac{P_k}{Q_k}$, $k = 1,\dots, \frac{n}{2}$ with
$\deg P_k=k+1$ and $\deg Q_k=k$.  The Kronecker indices are exactly $\deg P_k + \deg Q_k$ (cf. Proposition \ref{pv} and discussion after Theorem \ref{Vorontsov}), namely
$$
1, 3,5, \dots, n-1.
$$

The singular set $\mathsf{Sing}_0\subset \goth t_n^*$  is defined by an irreducible polynomial $\mathsf{f}$ of degree $n$. Therefore,  $\goth t_n$ possesses $n$ distinct characteristic numbers, each  of multiplicity one.  In particular, the Jordan part of a generic pencil $\mathcal A_{x+\lambda a}$ is diagonalisable.

A complete family of polynomials in bi-involution is formed by the ``shifts'' $P_k(x+\lambda a)$, $Q_k(x+\lambda a)$ and $\mathsf{f}(x+\lambda a)$.

If $n$ is odd, then $\goth t_n$ is of Kronecker type and the Kronecker indices are
$1,3,5, \dots, n$.

\subsection{Lie algebras with arbitrarily given JK invariants}

Let $\mathcal P =\{ \mathcal A + \lambda \mathcal B\}$ be an arbitrary pencil of skew-symmetric bi-linear forms.  
A natural question to ask is whether   $\mathcal P$ can be realised as a generic pencil $\mathcal A_{x+\lambda a}$ for a suitable Lie algebra $\goth g$?  In other words, we  want to describe all admissible 
Jordan-Kronecker invariants  of  finite dimensional  Lie algebras. 

First of all, notice that the JK invariants of a direct sum $\goth g_1 \oplus \goth g_2$ can naturally be obtained from those of $\goth g_1$ and $\goth g_2$ by ``summation''. In particular, the set of characteristic numbers for $\goth g_1 \oplus \goth g_2$ can be understood as the {\it disjoint} union of the corresponding sets for $\goth g_1$ and  $\goth g_2$. Thus,  first it is natural to study the realisation problem for the following simplest cases:
\begin{itemize}
\item a single Kronecker block,
\item a single $\lambda$-block  which consists of several Jordan blocks.  
\end{itemize}

\noindent Examples of such Lie algebras were constructed and communicated to us by I.~Kozlov \cite{kozlov2}.  

The first case can be realised by the Lie algebra $\goth g$ 
 with the basis $e_1,\dots  e_k$, $f_1,\dots, f_{k+1}$ and commutation relations:
$$
[e_{l}, f_i]= f_{i}, \quad [e_i, f_{i+1}]=f_{i+1}, \qquad i=1,\dots, k.
$$
This Lie algebra  admits the following matrix representation
$$
\begin{pmatrix}
A & b\\ 0 & 0
\end{pmatrix} \in \mathrm{gl}(k+2,\mathbb C),
$$
where $A$ denotes the matrix $\mathrm{diag}(a_1, a_2-a_1, a_3-a_2,\dots, a_{k}-a_{k-1}, -a_k)$, i.e., an arbitrary diagonal matrix with zero trace, and $b$ is a column of length $k+1$ with arbitrary entries.

 The index of $\goth g$ equals 1. The singular set $\mathsf{Sing}$ consists of several connected components each of which has codimension 2 and is defined by two linear equations $f_i=0, f_j=0$, $i\ne j$.  The Casimir function of the Lie-Poisson bracket on $\goth g^*$ is $f_1f_2\cdot\ldots\cdot f_{k+1}$.

The second case can be realised by the following matrix Lie algebra
$$
\goth g = \left\{ \begin{pmatrix}
a_0 &  x_1  & x_2   & \dots & x_m & b_0\\
        &  A_1 &  0      & \dots  &  0     & y_1\\
        &          &  A_2 & \ddots  &  \vdots     &  \vdots\\
        &          &           & \ddots  &   0     &  y_{m-1}\\
        &          &           &              &   A_m   & y_m \\
  0    &     0   &   0     &  \dots   &    0        &   0 
 \end{pmatrix}\right\}.
$$
Here  $x_k$ is an arbitrary row of length $n_k$,  $y_k$ is an arbitrary column of length $n_k$,  and $A_k$ is the $n_k\times n_k$-matrix related to the row $x_k = (x_k^1, \dots , x_k^{n_k})$  in the following way:
$$
A_k =
\begin{pmatrix}
a_0 & x_k^1 & x_k^2 & \dots & \!\!\! x_k^{n_k-2} & \!\!\! x_k^{n_k-1} \\  
        &  a_0  &  x_k^1 & \ddots &            & \!\!\! x_k^{n_k-2} \\
        &           &  a_0     &  \ddots           &                       & \!\!\!\!\!\! \vdots \\
        &           &              & \ddots &      \!\!\!   x_k^1               &\!\!\! \!\!\! x_k^2\\
        &           &              &             &    \!\!\!     a_0               & \!\!\!\!\!\! x_k^1\\
       &            &              &             &                        &  \!\!\!\!\!\!   a_0      
\end{pmatrix}
$$

This Lie algebra is Frobenius, its singular set $\mathsf{Sing} \subset \goth g^*$ is defined by the linear equation $f_0$=0, where $f_0\in \goth g$ is the matrix whose entries are all zero except for $b_0=1$ in the upper right corner. Let $n_1 = \max_{k=1,\dots, m} n_k$. Then the Jordan-Kronecker decomposition of a generic pencil $\mathcal A_{x+\lambda a}$ consists of Jordan blocks of sizes $2(n_1+1), 2n_2,\dots,  2n_m$.

Notice that the sizes of these Jordan blocks can be arbitrary with the only restriction.  Namely,  the largest Jordan block is unique, as by construction $n_1+1 > n_k$.  This restriction turns out to be a general property of non-degenerate Poisson pencils with non-constant characteristic numbers (see \cite{turiel}) and, therefore, is unavoidable.   For example, there is no Frobenius Lie algebra with diagonalisable $\lambda$-blocks if the multiplicity of $\lambda$ is greater than 1. 

However this restriction disappears if we allow Kronecker blocks. The simplest example which illustrates this phenomenon is the Heisenberg algebra with the basis $e_i$, $f_i$, $h$ $(i=1,\dots, n)$
and relations $[e_i, f_j] = \delta_{ij} h$.  A generic pencil $\mathcal A_{x+\lambda a}$ consists of one trivial Kronecker block and $n$  Jordan $2\times 2$ blocks with the same characteristic number $\lambda = -\frac{\langle h, x\rangle}{\langle h, a \rangle}$.

We hope that these observations will help to solve the realisation problem completely, but  so far this problem  remains open. The difficulty consists in non-trivial relations between Casimir functions and characteristic numbers.  By ``non-trivial''  we mean that the characteristic numbers can, in general,  be functionally dependent of the Casimir functions.  If it is not the case, then the splitting theorem recently proved by F.-J.Turiel \cite{tur3} implies that the Jordan-Kronecker invariants of a finite-dimensional Lie algebra $\goth g$ obey the restriction described above: for each characteristic number, the  largest Jordan block is unique.

\subsection{Lie algebras of low dimension}

The JK invariants for Lie algebras of low dimension  $\le 5$ (the list of such Lie algebras with some additional useful information can be found in   \cite{patera} and \cite{korot}) have been explicitly computed by Pumei Zhang \cite{PumeiDiss}.  The result is presented in the Appendix.

\subsection{Two examples of Frobenius Lie algebras}
 
The first example is  the Lie algebra  $\goth{aff}(n) = \mathrm{gl}(n) + \R^n$ of the group of affine transformations.  This Lie algebra is Frobenius and, therefore,  $\goth{aff}(n)$ is of Jordan type. To determine the sizes of Jordan blocks,  we need to describe the structure of the singular set.
It can be shown that $\Sing$ is defined by one irreducible polynomial
of degree  $\frac{1}{2}\dim \goth{aff}(n)$  (this polynomial is exactly the Pfaffian of the form $\mathcal A_x=\Bigl(\sum c_{ij}^k x_k\Bigr)$ which can, in fact, be written in much nicer form \eqref{affinv}, see \cite{PumeiDiss} for details).  Hence, by Theorem \ref{thfrob},  this Lie algebra has  $\frac12 \dim\goth{aff}(n)$ distinct characteristic numbers.  Each of them has multiplicity 1, i.e., a generic pencil $\mathcal A_{x+\lambda a}$ is diagonalisable, i.e., the sizes of Jordan blocks for $\goth{aff}(n)$ are
    $$
    \underbrace{2, \, 2, \, \dots \, , \, 2}_{\mbox{$k$ times}}, \quad k=\frac12(n^2+n)=\frac12\dim \goth{aff}(n).
    $$

Another interesting example is $\goth g = \mathrm{gl}(n) + \R^{n^2}$, where the vector space $\R^{n^2}$ is realised by $n\times n$ matrices, and the action of $\mathrm{gl}(n)$ on it is left multiplication. The matrix realisation of  $\goth g$ is as follows:
    $$
    \begin{pmatrix}
    A & C \\ 0 & 0
    \end{pmatrix},
    $$
where all entries are  $n\times n$ blocks,  and $A$ and $C$ are arbitrary.  The index of $\goth g$ is zero and, therefore, this Lie algebra is of Jordan type. The set of singular elements is defined (after natural, but not invariant identification of $\goth g$ and $\goth g^*$ by means of the pairing $\langle  (A_1,C_1), (A_2, C_2) \rangle = \Tr \, A_1A_2 + \Tr\, C_1C_2$, $(A_i, C_i)\in\goth g$) by the equation\footnote{The Pfaffian of $\mathcal A_x$ in this case is $\mathsf{f}(x)=\bigl(\mathsf{f}_{\mathrm{red}}(x)\bigr)^n$}
$$
\mathsf{f}_{\mathrm{red}}(x)=\det C =0, \quad  x=(A,C)\in\goth g^*.
$$

Since the (geometric) degree of $\Sing$ is $n$, there are $n$ distinct characteristic numbers $\lambda_1, \dots, \lambda_n$.  Moreover,  the irreducibility of $\Sing$ implies (Proposition  \ref{reducible}, item 3) that there is no essential difference between them so that all of them have the same multiplicity $n$ and the sizes of Jordan blocks are the same for  each $\lambda_i$.

To compute the sizes of Jordan blocks, it is sufficient to compute the annihilator of a typical singular point $y\in\Sing\subset \goth g^*$.
Straightforward computation shows that  $\dim\Ann y = 2n-2$.  Hence (see Proposition  \ref{predlast}) we have $n-1$ Jordan blocks and there is only one possibility for their sizes, namely\footnote{This, by the way, automatically implies $\ind\Ann y = 2 > \ind\goth g$.}:
$$
\underbrace{2, \, 2, \, \dots \, , \, 2}_{\mbox{$n-2$}}, 4.
$$

\subsection{Generalised argument shift conjecture}

For all Lie algebra listed above, the generalised argument shift conjecture holds true.  

In the semisimple case,  this follows from the Mischenko-Fomenko theorem \cite{MischFom}.  For the semidirect sums
$\goth g +_\phi V$, where $\goth g$ is simple and $\phi: \goth g \to \mathrm{End} (V)$ is irreducible,  we use the fact that all of them are of Kronecker type.  So in these two cases, the algebra of shifts $\mathcal F_a$ is complete and in-bi-involution.

The Lie algebras of dimension $\le 5$ have been studied case by case in \cite{PumeiDiss} and  explicit description of complete sets $\mathcal G_a$  of polynomials  in bi-involution are indicated in the Appendix.

The Lie algebra $\goth{aff}(n)$ is more interesting. To describe $\mathcal G_a$ explicitly, we will need an explicit formula for the polynomial $\mathsf{f}$ that defines the singular set. To that end,  we use the standard matrix realisation of $\goth{aff}(n)$:
$$
\begin{pmatrix}
M & v \\ 0_n & 0
\end{pmatrix}
$$
where $M$ is an arbitrary $n\times n$ matrix, and $v$ is a vector-column of length $n$. If we identify  this Lie algebra $\goth{aff}(n)$  with its dual space $\goth{aff}(n)^*$ by means of (non-invariant) pairing
$$
\langle (M_1, v_1), (M_2, v_2) \rangle = \Tr\, M_1 M_2 + \Tr \, v^\top_1 v_2
$$
then $\Sing$ can be defined by the equation $\mathsf{f}(x)=0$, where
\begin{equation}
\label{affinv}
\mathsf{f}(x)=\det (v, Mv, M^2 v, \dots, M^{n-1} v),  \quad x=(M,v)\in \goth{aff}^*(n).
\end{equation}

By using Theorem  \ref{thfrob}, we get

\begin{proposition}
For the Lie algebra $\goth{aff}(n)$, the generalised argument shift conjecture holds true.  As a complete family of polynomials in bi-involution we can take the coefficients of the expansion of 
$\mathsf{f}(x + \lambda a)$ into powers of  $\lambda$, where $\mathsf{f}$ is given by \eqref{affinv}.
\end{proposition}

In the case $\goth g= \mathrm{gl}(n)+\mathbb R^{n^2}$, this method does not work as the characteristic numbers have multiplicity $n$. But in this case the ideal $\goth h=\mathbb R^{n^2}\subset 
 \mathrm{gl}(n)+\mathbb R^{n^2}$ is commutative and therefore $P(\goth h)\subset P(\goth g)$ can be taken as the desired  algebra $\mathcal G_a$ of polynomials in bi-involution. The completeness is obvious as $n^2$ is exactly $\frac{1}{2}(\dim \goth g +\ind\goth g)$.




\newpage

\section{Appendix}


\vskip 40pt

\hskip-40pt \footnotesize{
\begin{tabular}{|c| c| c| c| c| c|}
\multicolumn{6}{c}{\normalsize Table of low-dimensional Lie algebras}\\[5pt]
\hline
 \begin{tabular}{l}\textbf{Name} \\ \textbf{and} \\ \textbf{Index} \end{tabular}  & \textbf{Relations} & \begin{tabular}{l}\textbf{Jordan--Kronecker} \\  \textbf{invariant} \end{tabular} & \begin{tabular}{l} \textbf{Char.}\\ \textbf{number} \end{tabular}  & \textbf{Singular set}  &  \textbf{Family $\mathcal{G}_a$}\\\hline\hline

& & & & &\\
$\begin{matrix}A_{3,1}\\
(\mathrm{ind}\, = 1)
\end{matrix}
$
 &
\begin{tabular}{l}
$[e_2,e_3]=e_1$
\end{tabular}
&
\begin{tabular}{l}
$\lambda$-block of size $2 {\times} 2$,  \\
 and\\
$\mathcal{K}$-block of size $1 {\times} 1$  
\end{tabular}
&
$
\begin{matrix}
\lambda=-\frac{x_1}{a_1}
\end{matrix}
$
&
$
\begin{matrix}
x_1=0, \\
\mathrm{codim}\, S=1
\end{matrix}
$
&
$x_1, x_2$\\
& & & & & \\[-2ex]\hline

& & & & &\\
$\begin{matrix}A_{3,2}\\
(\mathrm{ind}\, = 1)
\end{matrix}
$
 &
\begin{tabular}{l}
$[e_1,e_3]=e_1$,\\
$[e_2,e_3]=e_1+e_2$
\end{tabular}
&
\begin{tabular}{l}
$\mathcal{K}$-block of size $3 {\times} 3$  
\end{tabular}
&

&
$
\begin{matrix}
x_1=0, \\
x_2=0,\\
\mathrm{codim}\, S=2
\end{matrix}
$
&
$x_1, x_2$\\
& & & & & \\[-2ex]\hline

& & & & &\\
$\begin{matrix}A_{3,3}\\
(\mathrm{ind}\, = 1)
\end{matrix}
$
 &
\begin{tabular}{l}
$[e_1,e_3]=e_1$,\\
$[e_2,e_3]=e_2$
\end{tabular}
&
\begin{tabular}{l}
$\mathcal{K}$-block of size $3 {\times} 3$  
\end{tabular}
&

&
$
\begin{matrix}
x_1=0, \\
x_2=0,\\
\mathrm{codim}\, S=2
\end{matrix}
$
&
$x_1, x_2$\\
& & & & & \\[-2ex]\hline

& & & & &\\
$\begin{matrix}A_{3,4}\\
(\mathrm{ind}\, = 1)
\end{matrix}
$
 &
\begin{tabular}{l}
$[e_1,e_3]=e_1$,\\
$[e_2,e_3]=-e_2$
\end{tabular}
&
\begin{tabular}{l}
$\mathcal{K}$-block of size $3 {\times} 3$  
\end{tabular}
&

&
$
\begin{matrix}
x_1=0, \\
x_2=0,\\
\mathrm{codim}\, S=2
\end{matrix}
$
&
$x_1, x_2$\\
& & & & & \\[-2ex]\hline

& & & & &\\
$\begin{matrix}A_{3,5}^a\\
(\mathrm{ind}\, = 1)
\end{matrix}
$
 &
\begin{tabular}{l}
$[e_1,e_3]=e_1$,\\
$[e_2,e_3]=ae_2$\\
$(0<|a|<1)$
\end{tabular}
&
\begin{tabular}{l}
$\mathcal{K}$-block of size $3 {\times} 3$  
\end{tabular}
&

&
$
\begin{matrix}
x_1=0, \\
x_2=0,\\
\mathrm{codim}\, S=2
\end{matrix}
$
&
$x_1, x_2$\\
& & & & & \\[-2ex]\hline

& & & & &\\
$\begin{matrix}A_{3,6}\\
(\mathrm{ind}\, = 1)
\end{matrix}
$
 &
\begin{tabular}{l}
$[e_1,e_3]=-e_2$,\\
$[e_2,e_3]=e_1$
\end{tabular}
&
\begin{tabular}{l}
$\mathcal{K}$-block of size $3 {\times} 3$  
\end{tabular}
&

&
$
\begin{matrix}
x_1=0, \\
x_2=0,\\
\mathrm{codim}\, S=2
\end{matrix}
$
&
$x_1, x_2$\\
& & & & & \\[-2ex]\hline

& & & & &\\
$\begin{matrix}A_{3,7}^a\\
(\mathrm{ind}\, = 1)
\end{matrix}
$
 &
\begin{tabular}{l}
$[e_1,e_3]=ae_1-e_2$,\\
$[e_2,e_3]=e_1+ae_2$\\
$(a>0)$
\end{tabular}
&
\begin{tabular}{l}
$\mathcal{K}$-block of size $3 {\times} 3$  
\end{tabular}
&

&
$
\begin{matrix}
x_1=0, \\
x_2=0\\
\mathrm{codim}\, S=2
\end{matrix}
$
&
$x_1, x_2$\\
& & & & & \\[-2ex]\hline

& & & & &\\
$\begin{matrix}A_{3,8}\\
(\mathrm{ind}\, = 1)
\end{matrix}
$
 &
\begin{tabular}{l}
$[e_1,e_3]=-2e_2$,\\
$[e_1,e_2]=e_1$,\\
$[e_2,e_3]=e_3$
\end{tabular}
&
\begin{tabular}{l}
$\mathcal{K}$-block of size $3 {\times} 3$  
\end{tabular}
&

&
$
\begin{matrix}
x_1=0, \\
x_2=0,\\
x_3=0,\\
\mathrm{codim}\, S=3
\end{matrix}
$
&
$
\begin{array}{c}
2(x_2^2+x_1x_3),\\ \\
2(a_3x_1+2a_2x_2+a_1x_3)
\end{array}$\\ \hline

& & & & &\\
$\begin{matrix}A_{3,9}\\
(\mathrm{ind}\, = 1)
\end{matrix}
$
 &
\begin{tabular}{l}
$[e_1,e_2]=e_3$,\\
$[e_2,e_3]=e_1$,\\
$[e_3,e_1]=e_2$
\end{tabular}
&
\begin{tabular}{l}
$\mathcal{K}$-block of size $3 {\times} 3$  
\end{tabular}
&

&
$
\begin{matrix}
x_1=0, \\
x_2=0,\\
x_3=0,\\
\mathrm{codim}\, S=3
\end{matrix}
$
&
$\begin{array}{c}
x_1^2+x_2^2+x_3^2,\\ \\
2(a_1x_1+a_2x_2+a_3x_3)
\end{array}$\\
& & & & & \\  \hline

\end{tabular}

\newpage


$\quad$
\vskip-80pt
\hskip-50pt {\footnotesize\begin{tabular}{|c| c| c| c| c| c|}\hline
 \begin{tabular}{l}\textbf{Name} \\ \textbf{and} \\ \textbf{Index} \end{tabular}  & \textbf{Relations} & \begin{tabular}{l}\textbf{Jordan--Kronecker} \\  \textbf{invariant} \end{tabular} & \begin{tabular}{l} \textbf{Char.}\\ \textbf{number} \end{tabular}  & \textbf{Singular set}  &  \textbf{Family $\mathcal{G}_a$}\\\hline\hline

 & & & & &\\
$\begin{matrix}A_{4,1}\\
(\mathrm{ind}\, = 2)
\end{matrix}
$
 &
\begin{tabular}{l}
$[e_2,e_4]=e_1$,\\
$[e_3,e_4]=e_2$
\end{tabular}
&
\begin{tabular}{l}
$\mathcal{K}$-block of size $3 {\times} 3$,  \\
$\mathcal{K}$-block of size $1 {\times} 1$  
\end{tabular}
&

&
$
\begin{matrix}
x_1=0, \\
x_2=0,\\
\mathrm{codim}\, S=2
\end{matrix}
$
&
$x_1, x_2 , x_3$\\
& & & & & \\[-2ex]\hline

 & & & & &\\
$\begin{matrix}A_{4,2}^a\\
(\mathrm{ind}\, = 2)\\
a\ne 0
\end{matrix}
$
 &
\begin{tabular}{l}
$[e_1,e_4]=ae_1$,\\
$[e_2,e_4]=e_2$,\\
$[e_3,e_4]=e_2+e_3$
\end{tabular}
&
\begin{tabular}{l}
$\mathcal{K}$-block of size $3 {\times} 3$,  \\
$\mathcal{K}$-block of size $1 {\times} 1$  
\end{tabular}
&

&
$
\begin{matrix}
x_1=0, \\
x_2=0,\\
x_3=0,\\
\mathrm{codim}\, S=3
\end{matrix}
$
&
$x_1, x_2 , x_3$\\
& & & & & \\[-2ex]\hline

 & & & & &\\
$\begin{matrix}A_{4,3}\\
(\mathrm{ind}\, = 2)
\end{matrix}
$
 &
\begin{tabular}{l}
$[e_1,e_4]=e_1$,\\
$[e_3,e_4]=e_2$
\end{tabular}
&
\begin{tabular}{l}
$\mathcal{K}$-block of size $3 {\times} 3$,  \\
$\mathcal{K}$-block of size $1 {\times} 1$  
\end{tabular}
&

&
$
\begin{matrix}
x_1=0, \\
x_2=0,\\
\mathrm{codim}\, S=2
\end{matrix}
$
&
$x_1, x_2 , x_3$\\
& & & & & \\[-2ex]\hline

 & & & & &\\
$\begin{matrix}A_{4,4}\\
(\mathrm{ind}\, = 2)
\end{matrix}
$
 &
\begin{tabular}{l}
$[e_1,e_4]=e_1$,\\
$[e_2,e_4]=e_1+e_2$,\\
$[e_3,e_4]=e_2+e_3$
\end{tabular}
&
\begin{tabular}{l}
$\mathcal{K}$-block of size $3 {\times} 3$,  \\
$\mathcal{K}$-block of size $1 {\times} 1$  
\end{tabular}
&

&
$
\begin{matrix}
x_1=0, \\
x_2=0,\\
x_3=0,\\
\mathrm{codim}\, S=3
\end{matrix}
$
&
$x_1, x_2 , x_3$\\
& & & & & \\[-2ex]\hline

 & & & & &\\
$\begin{matrix}A_{4,5}^{ab}\\
(\mathrm{ind}\, = 2)\\
(ab\ne0)
\end{matrix}
$
 &
\begin{tabular}{l}
$[e_1,e_4]=e_1$,\\
$[e_2,e_4]=ae_2$,\\
$[e_3,e_4]=be_3$,\\
$(-1\le a \le b \le 1)$
\end{tabular}
&
\begin{tabular}{l}
$\mathcal{K}$-block of size $3 {\times} 3$,  \\
$\mathcal{K}$-block of size $1 {\times} 1$  
\end{tabular}
&

&
$
\begin{matrix}
x_1=0, \\
x_2=0,\\
x_3=0,\\
\mathrm{codim}\, S=3
\end{matrix}
$
&
$x_1, x_2 , x_3$\\
& & & & & \\[-2ex]\hline

 & & & & &\\
$\begin{matrix}A_{4,6}^{ab}\\
(\mathrm{ind}\, = 2)\\
(a\ne 0),\\
(b\ge 0)
\end{matrix}
$
 &
\begin{tabular}{l}
$[e_1,e_4]=ae_1$,\\
$[e_2,e_4]=be_2-e_3$,\\
$[e_3,e_4]=e_2+be_3$
\end{tabular}
&
\begin{tabular}{l}
$\mathcal{K}$-block of size $3 {\times} 3$,  \\
$\mathcal{K}$-block of size $1 {\times} 1$  
\end{tabular}
&

&
$
\begin{matrix}
x_1=0, \\
x_2=0,\\
x_3=0,\\
\mathrm{codim}\, S=3
\end{matrix}
$
&
$x_1, x_2 , x_3$\\\hline

& & & & &\\
$\begin{matrix}A_{4,7}\\
(\mathrm{ind}\, = 0)
\end{matrix}
$
 &
$
\begin{tabular}{l}
$[e_2,e_3]=e_1$,\\
$[e_1,e_4]=2e_1$,\\
$[e_2,e_4]=e_2$,\\
$[e_3,e_4]=e_2+e_3$
\end{tabular}
$ 
&
\begin{tabular}{l}
$\lambda$-block of size $4 {\times} 4$ 
\end{tabular}
&
$
\begin{matrix}
\lambda=-\frac{x_1}{a_1}
\end{matrix}
$
&
$
\begin{matrix}
x_1=0, \\
\mathrm{codim}\, S=1
\end{matrix}
$
&
$x_1, x_2$\\
& & & & & \\[-2ex]\hline

 & & & & &\\
$\begin{matrix}A_{4,8}\\
(\mathrm{ind}\, = 2)
\end{matrix}
$
 &
\begin{tabular}{l}
$[e_2,e_3]=e_1$,\\
$[e_2,e_4]=e_2$,\\
$[e_3,e_4]=-e_3$
\end{tabular}
&
\begin{tabular}{l}
$\mathcal{K}$-block of size $3 {\times} 3$,  \\
$\mathcal{K}$-block of size $1 {\times} 1$  
\end{tabular}
&

&
$
\begin{matrix}
x_1=0, \\
x_2=0,\\
x_3=0,\\
\mathrm{codim}\, S=3
\end{matrix}
$
&
$
\begin{array}{c}
x_1,\\
2(x_2x_3-x_1x_4),\\
2(-a_4x_1+a_3x_2+\\
+a_2x_3-a_1x_4)
\end{array}
$\\
& & & & & \\[-2ex]\hline

& & & & &\\
$\begin{matrix}A_{4,9}^b\\
(\mathrm{ind}\, = 0)\\
(-1<b\le 1)
\end{matrix}
$
 &
$
\begin{tabular}{l}
$[e_2,e_3]=e_1$,\\
$[e_1,e_4]=(1+b)e_1$,\\
$[e_2,e_4]=e_2$,\\
$[e_3,e_4]=be_3$,
\end{tabular}
$ 
&
\begin{tabular}{l}
$\lambda$-block of size $4 {\times} 4$ 
\end{tabular}
&
$
\begin{matrix}
\lambda=-\frac{x_1}{a_1}
\end{matrix}
$
&
$
\begin{matrix}
x_1=0, \\
\mathrm{codim}\, S=1
\end{matrix}
$
&
$x_1, x_2$\\
& & & & & \\[-2ex]\hline

 & & & & &\\
$\begin{matrix}A_{4,10}\\
(\mathrm{ind}\, = 2)
\end{matrix}
$
 &
\begin{tabular}{l}
$[e_2,e_3]=e_1$,\\
$[e_2,e_4]=-e_3$,\\
$[e_3,e_4]=e_2$
\end{tabular}
&
\begin{tabular}{l}
$\mathcal{K}$-block of size $3 {\times} 3$,  \\
$\mathcal{K}$-block of size $1 {\times} 1$  
\end{tabular}
&

&
$
\begin{matrix}
x_1=0, \\
x_2=0,\\
x_3=0,\\
\mathrm{codim}\, S=3
\end{matrix}
$
&
$
\begin{array}{c}
x_1,\\ 
2x_1x_4+x_2^2+x_3^2,\\ 
2(a_4x_1+a_2x_2+\\
+a_3x_3+a_1x_4)
\end{array}
$\\
& & & & & \\[-2ex]\hline

& & & & &\\
$\begin{matrix}A_{4,11}^a\\
(\mathrm{ind}\, = 0)
\end{matrix}
$
 &
$
\begin{tabular}{l}
$[e_2,e_3]=e_1$,\\
$[e_1,e_4]=2ae_1$,\\
$[e_2,e_4]=ae_2-e_3$,\\
$[e_3,e_4]=e_2+ae_3$,\\
$(a>0)$
\end{tabular}
$ 
&
\begin{tabular}{l}
$\lambda$-block of size $4 {\times} 4$ 
\end{tabular}
&
$
\begin{matrix}
\lambda=-\frac{x_1}{a_1}
\end{matrix}
$
&
$
\begin{matrix}
x_1=0, \\
\mathrm{codim}\, S=1
\end{matrix}
$
&
$x_1, x_2$\\
& & & & & \\[-2ex]\hline

& & & & &\\
$\begin{matrix}A_{4,12}\\
(\mathrm{ind}\, = 0)
\end{matrix}
$
 &
$
\begin{tabular}{l}
$[e_1,e_3]=e_1$,\\
$[e_2,e_3]=e_2$,\\
$[e_1,e_4]=-e_2$,\\
$[e_2,e_4]=e_1$
\end{tabular}
$ 
&
\begin{tabular}{l}
$\lambda_1$-block of size $2 {\times} 2$,  \\
$\lambda_2$-block of size $2  {\times} 2$,  \\
$\lambda_1\ne \lambda_2$ 
\end{tabular}
&
$
\begin{matrix}
\lambda_1=-\frac{x_1+ix_2}{a_1+ia_2}\\
\lambda_2=-\frac{x_1-ix_2}{a_1-ia_2}
\end{matrix}
$
&
$
\begin{matrix}
x_1^2+x_2^2=0, \\
\mathrm{codim}\, S=1
\end{matrix}
$
&
$x_1, x_2$\\\hline
\end{tabular}

$\quad$
\vskip-80pt
\hskip-45pt \footnotesize{\begin{tabular}{|c| c| c| c| c| c|}\hline
 \begin{tabular}{l}\textbf{Name} \\ \textbf{and} \\ \textbf{Index} \end{tabular}  & \textbf{Relations} & \begin{tabular}{l}\textbf{Jordan--Kronecker} \\  \textbf{invariant} \end{tabular} & \begin{tabular}{l} \textbf{Char.}\\ \textbf{number} \end{tabular}  & \textbf{Singular set}  &  \textbf{Family $\mathcal{G}_a$}\\\hline\hline

 & & & & &\\
$\begin{matrix}A_{5,1}\\
(\mathrm{ind}\, = 3)
\end{matrix}
$
 &
\begin{tabular}{l}
$[e_3,e_5]=e_1$,\\
$[e_4,e_5]=e_2$
\end{tabular}
&
\begin{tabular}{l}
$\mathcal{K}$-block of size $3 {\times} 3$,  \\
$\mathcal{K}$-block of size $1 {\times} 1$, \\
$\mathcal{K}$-block of size $1 {\times} 1$
\end{tabular}
&

&
$
\begin{matrix}
x_1=0, \\
x_2=0,\\
\mathrm{codim}\, S=2
\end{matrix}
$
&
$x_1,x_2,x_3,x_4$
\\
& & & & & \\[-2ex]\hline

 & & & & &\\
$\begin{matrix}A_{5,2}\\
(\mathrm{ind}\, = 3)
\end{matrix}
$
 &
\begin{tabular}{l}
$[e_2,e_5]=e_1$,\\
$[e_3,e_5]=e_2$,\\
$[e_4,e_5]=e_3$
\end{tabular}
&
\begin{tabular}{l}
$\mathcal{K}$-block of size $3 {\times} 3$,  \\
$\mathcal{K}$-block of size $1 {\times} 1$,  \\
$\mathcal{K}$-block of size $1 {\times} 1$
\end{tabular}
&

&
$
\begin{matrix}
x_1=0, \\
x_2=0,\\
x_3=0,\\
\mathrm{codim}\, S=3
\end{matrix}
$
&
$x_1,x_2,x_3,x_4$
\\
& & & & & \\[-2ex]\hline

 & & & & &\\
$\begin{matrix}A_{5,3}\\
(\mathrm{ind}\, = 3)
\end{matrix}
$
 &
\begin{tabular}{l}
$[e_3,e_4]=e_2$,\\
$[e_3,e_5]=e_1$,\\
$[e_4,e_5]=e_3$
\end{tabular}
&
\begin{tabular}{l}
$\mathcal{K}$-block of size $3 {\times} 3$,  \\
$\mathcal{K}$-block of size $1 {\times} 1$,  \\
$\mathcal{K}$-block of size $1 {\times} 1$
\end{tabular}
&

&
$
\begin{matrix}
x_1=0, \\
x_2=0,\\
x_3=0,\\
\mathrm{codim}\, S=3
\end{matrix}
$
&
$
\hskip-10pt\begin{array}{c}
x_1, x_2,\\ 
x_3^2+2x_2x_5-2x_1x_4,\\ 
2(-a_4x_1+a_5x_2+\\
+a_3x_3-a_1x_4+a_2x_5)
\end{array}
$
\\
& & & & & \\[-2ex]\hline

& & & & &\\
$\begin{matrix}A_{5,4}\\
(\mathrm{ind}\, = 1)
\end{matrix}
$
 &
$
\begin{tabular}{l}
$[e_2,e_4]=e_1$,\\
$[e_3,e_5]=e_1$
\end{tabular}
$ 
&
\begin{tabular}{l}
$\lambda$-block of size $2 {\times} 2$,  \\
$\lambda$-block of size $2  {\times} 2$,  \\
 and\\
$\mathcal{K}$-block of size $1 {\times} 1$  
\end{tabular}
&
$
\begin{matrix}
\lambda=-\frac{x_1}{a_1}
\end{matrix}
$
&
$
\begin{matrix}
x_1=0, \\
\mathrm{codim}\, S=1
\end{matrix}
$
&
$x_1, x_2, x_3$\\
& & & & & \\[-2ex]\hline

& & & & &\\
$\begin{matrix}A_{5,5}\\
(\mathrm{ind}\, = 1)
\end{matrix}
$
 &
$
\begin{tabular}{l}
$[e_3,e_4]=e_1$,\\
$[e_2,e_5]=e_1$,\\
$[e_3,e_5]=e_2$
\end{tabular}
$ 
&
\begin{tabular}{l}
$\lambda$-block of size $4  {\times} 4$,  \\
 and\\
$\mathcal{K}$-block of size $1 {\times} 1$  
\end{tabular}
&
$
\begin{matrix}
\lambda=-\frac{x_1}{a_1}
\end{matrix}
$
&
$
\begin{matrix}
x_1=0, \\
\mathrm{codim}\, S=1
\end{matrix}
$
&
$x_1, x_2, x_3$\\
& & & & & \\[-2ex]\hline

& & & & &\\
$\begin{matrix}A_{5,6}\\
(\mathrm{ind}\, = 1)
\end{matrix}
$
 &
$
\begin{tabular}{l}
$[e_3,e_4]=e_1$,\\
$[e_2,e_5]=e_1$,\\
$[e_3,e_5]=e_2$,\\
$[e_4,e_5]=e_3$
\end{tabular}
$ 
&
\begin{tabular}{l}
$\lambda$-block of size $4  {\times} 4$,  \\
 and\\
$\mathcal{K}$-block of size $1 {\times} 1$  
\end{tabular}
&
$
\begin{matrix}
\lambda=-\frac{x_1}{a_1}
\end{matrix}
$
&
$
\begin{matrix}
x_1=0, \\
\mathrm{codim}\, S=1
\end{matrix}
$
&
$x_1, x_2, x_3$\\
& & & & & \\[-2ex]\hline

 & & & & &\\
$\begin{matrix}A_{5,7}^{abc}\\
(\mathrm{ind}\, = 3)
\end{matrix}
$
 &
\begin{tabular}{l}
$[e_1,e_5]=e_1$,\\
$[e_2,e_5]=ae_2$,\\
$[e_3,e_5]=be_3$,\\
$[e_4,e_5]=ce_4$,\\
$(abc\ne0)$,\\
$(-1\le c\le b\le a\le 1)$
\end{tabular}
&
\begin{tabular}{l}
$\mathcal{K}$-block of size $3 {\times} 3$,  \\
$\mathcal{K}$-block of size $1 {\times} 1$,  \\
$\mathcal{K}$-block of size $1 {\times} 1$
\end{tabular}
&

&
$
\begin{matrix}
x_1=0, \\
x_2=0,\\
x_3=0,\\
x_4=0,\\
\mathrm{codim}\, S=4
\end{matrix}
$
&
$
x_1,x_2,x_3,x_4
$
\\\hline

 & & & & &\\
$\begin{matrix}A_{5,8}^{c}\\
(\mathrm{ind}\, = 3)
\end{matrix}
$
 &
\begin{tabular}{l}
$[e_2,e_5]=e_1$,\\
$[e_3,e_5]=e_3$,\\
$[e_4,e_5]=ce_4$,\\
$(-1< |c|\le 1)$
\end{tabular}
&
\begin{tabular}{l}
$\mathcal{K}$-block of size $3 {\times} 3$,  \\
$\mathcal{K}$-block of size $1 {\times} 1$,  \\
$\mathcal{K}$-block of size $1 {\times} 1$
\end{tabular}
&

&
$
\begin{matrix}
x_1=0, \\
x_3=0,\\
x_4=0,\\
\mathrm{codim}\, S=3
\end{matrix}
$
&
$
x_1,x_2,x_3,x_4
$
\\
& & & & & \\[-2ex]\hline

 & & & & &\\
$\begin{matrix}A_{5,9}^{bc}\\
(\mathrm{ind}\, = 3)
\end{matrix}
$
 &
\begin{tabular}{l}
$[e_1,e_5]=e_1$,\\
$[e_2,e_5]=e_1+e_2$,\\
$[e_3,e_5]=be_3$,\\
$[e_4,e_5]=ce_4$,\\
$(0\ne c\le b)$
\end{tabular}
&
\begin{tabular}{l}
$\mathcal{K}$-block of size $3 {\times} 3$,  \\
$\mathcal{K}$-block of size $1 {\times} 1$,  \\
$\mathcal{K}$-block of size $1 {\times} 1$
\end{tabular}
&

&
$
\begin{matrix}
x_1=0, \\
x_2=0,\\
x_3=0,\\
x_4=0,\\
\mathrm{codim}\, S=4
\end{matrix}
$
&
$
x_1,x_2,x_3,x_4
$
\\
& & & & & \\[-2ex]\hline

 & & & & &\\
$\begin{matrix}A_{5,10}\\
(\mathrm{ind}\, = 3)
\end{matrix}
$
 &
\begin{tabular}{l}
$[e_2,e_5]=e_1$,\\
$[e_3,e_5]=e_2$,\\
$[e_4,e_5]=e_4$
\end{tabular}
&
\begin{tabular}{l}
$\mathcal{K}$-block of size $3 {\times} 3$,  \\
$\mathcal{K}$-block of size $1 {\times} 1$,  \\
$\mathcal{K}$-block of size $1 {\times} 1$
\end{tabular}
&

&
$
\begin{matrix}
x_1=0, \\
x_2=0,\\
x_4=0,\\
\mathrm{codim}\, S=3
\end{matrix}
$
&
$
x_1,x_2,x_3,x_4
$
\\
& & & & & \\[-2ex]\hline

 & & & & &\\
$\begin{matrix}A_{5,11}^{c}\\
(\mathrm{ind}\, = 3)
\end{matrix}
$
 &
\begin{tabular}{l}
$[e_1,e_5]=e_1$,\\
$[e_2,e_5]=e_1+e_2$,\\
$[e_3,e_5]=e_2+e_3$,\\
$[e_4,e_5]=ce_4$,\\
$(c\ne0)$
\end{tabular}
&
\begin{tabular}{l}
$\mathcal{K}$-block of size $3 {\times} 3$,  \\
$\mathcal{K}$-block of size $1 {\times} 1$,  \\
$\mathcal{K}$-block of size $1 {\times} 1$
\end{tabular}
&

&
$
\begin{matrix}
x_1=0, \\
x_2=0,\\
x_3=0,\\
x_4=0,\\
\mathrm{codim}\, S=4
\end{matrix}
$
&
$
x_1,x_2,x_3,x_4
$
\\
& & & & & \\  \hline

\end{tabular}

\newpage


$\quad 4$
\vskip-80pt
\hskip-30pt
\footnotesize{\begin{tabular}{|c| c| c| c| c| c|}\hline
 \begin{tabular}{l}\textbf{Name} \\ \textbf{and} \\ \textbf{Index} \end{tabular}  & \textbf{Relations} & \begin{tabular}{l}\textbf{Jordan--Kronecker} \\  \textbf{invariant} \end{tabular} & \begin{tabular}{l} \textbf{Char.}\\ \textbf{number} \end{tabular}  & \textbf{Singular set}  &  \textbf{Family $\mathcal{G}_a$}\\\hline\hline

 & & & & &\\
$\begin{matrix}A_{5,12}\\
(\mathrm{ind}\, = 3)
\end{matrix}
$
 &
\begin{tabular}{l}
$[e_1,e_5]=e_1$,\\
$[e_2,e_5]=e_1+e_2$,\\
$[e_3,e_5]=e_2+e_3$,\\
$[e_4,e_5]=e_3+e_4$
\end{tabular}
&
\begin{tabular}{l}
$\mathcal{K}$-block of size $3 {\times} 3$,  \\
$\mathcal{K}$-block of size $1 {\times} 1$,  \\
$\mathcal{K}$-block of size $1 {\times} 1$
\end{tabular}
&

&
$
\begin{matrix}
x_1=0, \\
x_2=0,\\
x_3=0,\\
x_4=0,\\
\mathrm{codim}\, S=4
\end{matrix}
$
&
$
x_1,x_2,x_3,x_4
$
\\
& & & & & \\[-2ex]\hline

 & & & & &\\
$\begin{matrix}A_{5,13}^{apq}\\
(\mathrm{ind}\, = 3)
\end{matrix}
$
 &
\begin{tabular}{l}
$[e_1,e_5]=e_1$,\\
$[e_2,e_5]=ae_2$,\\
$[e_3,e_5]=pe_3-qe_4$,\\
$[e_4,e_5]=qe_3+pe_4$,\\
$(aq\ne0, |a|\le 1)$ \\

\end{tabular}
&
\begin{tabular}{l}
$\mathcal{K}$-block of size $3 {\times} 3$,  \\
$\mathcal{K}$-block of size $1 {\times} 1$,  \\
$\mathcal{K}$-block of size $1 {\times} 1$
\end{tabular}
&

&
$
\begin{matrix}
x_1=0, \\
x_2=0,\\
x_3=0,\\
x_4=0,\\
\mathrm{codim}\, S=4
\end{matrix}
$
&
$
x_1,x_2,x_3,x_4
$
\\ \hline

 & & & & &\\
$\begin{matrix}A_{5,14}^{p}\\
(\mathrm{ind}\, = 3)
\end{matrix}
$
 &
\begin{tabular}{l}
$[e_2,e_5]=e_1$,\\
$[e_3,e_5]=pe_3-e_4$,\\
$[e_4,e_5]=e_3+pe_4$
\end{tabular}
&
\begin{tabular}{l}
$\mathcal{K}$-block of size $3 {\times} 3$,  \\
$\mathcal{K}$-block of size $1 {\times} 1$,  \\
$\mathcal{K}$-block of size $1 {\times} 1$
\end{tabular}
&

&
$
\begin{matrix}
x_1=0, \\
x_3=0,\\
x_4=0,\\
\mathrm{codim}\, S=3
\end{matrix}
$
&
$
x_1,x_2,x_3,x_4
$
\\
& & & & & \\[-2ex]\hline

 & & & & &\\
$\begin{matrix}A_{5,15}^{a}\\
(\mathrm{ind}\, = 3)
\end{matrix}
$
 &
\begin{tabular}{l}
$[e_1,e_5]=e_1$,\\
$[e_2,e_5]=e_1+e_2$,\\
$[e_3,e_5]=ae_3$,\\
$[e_4,e_5]=e_3+ae_4$,\\
$(|a|\le 1)$
\end{tabular}
&
\begin{tabular}{l}
$\mathcal{K}$-block of size $3 {\times} 3$,  \\
$\mathcal{K}$-block of size $1 {\times} 1$,  \\
$\mathcal{K}$-block of size $1 {\times} 1$
\end{tabular}
&

&
$
\begin{matrix}
x_1=0, \\
x_2=0,\\
x_3=0,\\
x_4=0,\\
\mathrm{codim}\, S=4
\end{matrix}
$
&
$
x_1,x_2,x_3,x_4
$
\\
& & & & & \\[-2ex]\hline

 & & & & &\\
$\begin{matrix}A_{5,16}^{pq}\\
(\mathrm{ind}\, = 3)
\end{matrix}
$
 &
\begin{tabular}{l}
$[e_1,e_5]=e_1$,\\
$[e_2,e_5]=e_1+e_2$,\\
$[e_3,e_5]=pe_3-qe_4$,\\
$[e_4,e_5]=qe_3+pe_4$,\\
$(q\ne0)$
\end{tabular}
&
\begin{tabular}{l}
$\mathcal{K}$-block of size $3 {\times} 3$,  \\
$\mathcal{K}$-block of size $1 {\times} 1$,  \\
$\mathcal{K}$-block of size $1 {\times} 1$
\end{tabular}
&

&
$
\begin{matrix}
x_1=0, \\
x_2=0,\\
x_3=0,\\
x_4=0,\\
\mathrm{codim}\, S=4
\end{matrix}
$
&
$
x_1,x_2,x_3,x_4
$
\\
& & & & & \\[-2ex]\hline

 & & & & &\\
$\begin{matrix}A_{5,17}^{spq}\\
(\mathrm{ind}\, = 3)
\end{matrix}
$
 &
\begin{tabular}{l}
$[e_1,e_5]=pe_1-e_2$,\\
$[e_2,e_5]=e_1+pe_2$,\\
$[e_3,e_5]=qe_3-se_4$,\\
$[e_4,e_5]=se_3+qe_4$,\\
$(s\ne0)$
\end{tabular}
&
\begin{tabular}{l}
$\mathcal{K}$-block of size $3 {\times} 3$,  \\
$\mathcal{K}$-block of size $1 {\times} 1$,  \\
$\mathcal{K}$-block of size $1 {\times} 1$
\end{tabular}
&

&
$
\begin{matrix}
x_1=0, \\
x_2=0,\\
x_3=0,\\
x_4=0,\\
\mathrm{codim}\, S=4
\end{matrix}
$
&
$
x_1,x_2,x_3,x_4
$
\\
& & & & & \\[-2ex]\hline

 & & & & &\\
$\begin{matrix}A_{5,18}^{p}\\
(\mathrm{ind}\, = 3)
\end{matrix}
$
 &
\begin{tabular}{l}
$[e_1,e_5]=pe_1-e_2$,\\
$[e_2,e_5]=e_1+pe_2$,\\
$[e_3,e_5]=e_1+pe_3-e_4$,\\
$[e_4,e_5]=e_2+e_3+pe_4$,\\
$(p\le 0)$
\end{tabular}
&
\begin{tabular}{l}
$\mathcal{K}$-block of size $3 {\times} 3$,  \\
$\mathcal{K}$-block of size $1 {\times} 1$,  \\
$\mathcal{K}$-block of size $1 {\times} 1$
\end{tabular}
&

&
$
\begin{matrix}
x_1=0, \\
x_2=0,\\
x_3=0,\\
x_4=0,\\
\mathrm{codim}\, S=4
\end{matrix}
$
&
$
x_1,x_2,x_3,x_4
$
\\
& & & & & \\[-2ex]\hline

& & & & &\\
$\begin{matrix}A_{5,19}^{ab}\\
(\mathrm{ind}\, = 1)\\
(b\ne0)
\end{matrix}
$
 &
$
\begin{tabular}{l}
$[e_2,e_3]=e_1$,\\
$[e_1,e_5]=ae_1$,\\
$[e_2,e_5]=e_2$,\\
$[e_3,e_5]=(a-1)e_3$,\\
$[e_4,e_5]=be_4$,
\end{tabular}
$ 
&
\begin{tabular}{l}
$\lambda$-block of size $2 {\times} 2$,  \\
and\\
$\mathcal{K}$-block of size $3 {\times} 3$  
\end{tabular}
&
$
\begin{matrix}
\lambda=-\frac{x_1}{a_1}
\end{matrix}
$
&
$
\begin{matrix}
x_1=0, \\
\mathrm{codim}\, S=1
\end{matrix}
$
&
$x_1, x_2, x_4$\\\hline

& & & & &\\
$\begin{matrix}A_{5,20}^{a}\\
(\mathrm{ind}\, = 1)
\end{matrix}
$
 &
$
\begin{tabular}{l}
$[e_2,e_3]=e_1$,\\
$[e_1,e_5]=ae_1$,\\
$[e_2,e_5]=e_2$,\\
$[e_3,e_5]=(a-1)e_3$,\\
$[e_4,e_5]=e_1+ae_4$
\end{tabular}
$ 
&
\begin{tabular}{l}
$\lambda$-block of size $2 {\times} 2$,  \\
and\\
$\mathcal{K}$-block of size $3 {\times} 3$  
\end{tabular}
&
$
\begin{matrix}
\lambda=-\frac{x_1}{a_1}
\end{matrix}
$
&
$
\begin{matrix}
x_1=0, \\
\mathrm{codim}\, S=1
\end{matrix}
$
&
$x_1, x_2, x_4$\\
& & & & & \\[-2ex]\hline

& & & & &\\
$\begin{matrix}A_{5,21}\\
(\mathrm{ind}\, = 1)
\end{matrix}
$
 &
$
\begin{tabular}{l}
$[e_2,e_3]=e_1$,\\
$[e_1,e_5]=2e_1$,\\
$[e_2,e_5]=e_2+e_3$,\\
$[e_3,e_5]=e_3+e_4$,\\
$[e_4,e_5]=e_4$
\end{tabular}
$ 
&
\begin{tabular}{l}
$\lambda$-block of size $2 {\times} 2$,  \\
and\\
$\mathcal{K}$-block of size $3 {\times} 3$  
\end{tabular}
&
$
\begin{matrix}
\lambda=-\frac{x_1}{a_1}
\end{matrix}
$
&
$
\begin{matrix}
x_1=0, \\
\mathrm{codim}\, S=1
\end{matrix}
$
&
$x_1, x_2, x_4$\\
& & & & & \\ \hline

\end{tabular}

\newpage


$\quad$
\vskip-60pt
\hskip-20pt
\footnotesize{
\begin{tabular}{|c| c| c| c| c| c|}\hline
 \begin{tabular}{l}\textbf{Name} \\ \textbf{and} \\ \textbf{Index} \end{tabular}  & \textbf{Relations} & \begin{tabular}{l}\textbf{Jordan--Kronecker} \\  \textbf{invariant} \end{tabular} & \begin{tabular}{l} \textbf{Char.}\\ \textbf{number} \end{tabular}  & \textbf{Singular set}  &  \textbf{Family $\mathcal{G}_a$}\\\hline\hline

& & & & &\\
$\begin{matrix}A_{5,22}\\
(\mathrm{ind}\, = 1)
\end{matrix}
$
 &
$
\begin{tabular}{l}
$[e_2,e_3]=e_1$,\\
$[e_2,e_5]=e_3$,\\
$[e_4,e_5]=e_4$
\end{tabular}
$ 
&
\begin{tabular}{l}
$\lambda_1$-block of size $2 {\times} 2$,  \\
$\lambda_2$-block of size $2  {\times} 2$,  \\
$\lambda_1\ne \lambda_2$ \\
 and\\
$\mathcal{K}$-block of size $1 {\times} 1$  
\end{tabular}
&
$
\begin{matrix}
\lambda_1=-\frac{x_1}{a_1}\\
\lambda_2=-\frac{x_4}{a_4}
\end{matrix}
$
&
$
\begin{matrix}
x_1x_4=0, \\
\mathrm{codim}\, S=1
\end{matrix}
$
&
$x_1, x_2, x_4$\\
& & & & & \\[-2ex]\hline

& & & & &\\
$\begin{matrix}A_{5,23}^{b}\\
(\mathrm{ind}\, = 1)
\end{matrix}
$
 &
$
\begin{tabular}{l}
$[e_2,e_3]=e_1$,\\
$[e_1,e_5]=2e_1$,\\
$[e_2,e_5]=e_2+e_3$,\\
$[e_3,e_5]=e_3$,\\
$[e_4,e_5]=be_4$,\\
$(b\ne0)$
\end{tabular}
$ 
&
\begin{tabular}{l}
$\lambda$-block of size $2 {\times} 2$,  \\
and\\
$\mathcal{K}$-block of size $3 {\times} 3$  
\end{tabular}
&
$
\begin{matrix}
\lambda=-\frac{x_1}{a_1}
\end{matrix}
$
&
$
\begin{matrix}
x_1=0, \\
\mathrm{codim}\, S=1
\end{matrix}
$
&
$x_1, x_2, x_4$\\
& & & & & \\[-2ex]\hline

& & & & &\\
$\begin{matrix}A_{5,24}^{\epsilon}\\
(\mathrm{ind}\, = 1)\\
(\epsilon = \pm 1)
\end{matrix}
$
 &
$
\begin{tabular}{l}
$[e_2,e_3]=e_1$,\\
$[e_1,e_5]=2e_1$,\\
$[e_2,e_5]=e_2+e_3$,\\
$[e_3,e_5]=e_3$,\\
$[e_4,e_5]=\epsilon e_1+2e_4$
\end{tabular}
$ 
&
\begin{tabular}{l}
$\lambda$-block of size $2 {\times} 2$,  \\
and\\
$\mathcal{K}$-block of size $3 {\times} 3$  
\end{tabular}
&
$
\begin{matrix}
\lambda=-\frac{x_1}{a_1}
\end{matrix}
$
&
$
\begin{matrix}
x_1=0, \\
\mathrm{codim}\, S=1
\end{matrix}
$
&
$x_1, x_2, x_4$\\ \hline

& & & & &\\
$\begin{matrix}A_{5,25}^{bp}\\
(\mathrm{ind}\, = 1)
\end{matrix}
$
 &
$
\begin{tabular}{l}
$[e_2,e_3]=e_1$,\\
$[e_1,e_5]=2pe_1$,\\
$[e_2,e_5]=pe_2+e_3$,\\
$[e_3,e_5]=pe_3-e_2$,\\
$[e_4,e_5]=be_4$,\\
$(b\ne 0)$
\end{tabular}
$ 
&
\begin{tabular}{l}
$\lambda$-block of size $2 {\times} 2$,  \\
and\\
$\mathcal{K}$-block of size $3 {\times} 3$  
\end{tabular}
&
$
\begin{matrix}
\lambda=-\frac{x_1}{a_1}
\end{matrix}
$
&
$
\begin{matrix}
x_1=0, \\
\mathrm{codim}\, S=1
\end{matrix}
$
&
$x_1, x_2, x_4$\\
& & & & & \\[-3ex]\hline

& & & & &\\
$\begin{matrix}A_{5,26}^{p\epsilon}\\
(\mathrm{ind}\, = 1)
\end{matrix}
$
 &
$
\begin{tabular}{l}
$[e_2,e_3]=e_1$,\\
$[e_1,e_5]=2pe_1$,\\
$[e_2,e_5]=pe_2+e_3$,\\
$[e_3,e_5]=pe_3-e_2$,\\
$[e_4,e_5]=\epsilon e_1+2e_4$,\\
$(\epsilon = \pm 1)$
\end{tabular}
$ 
&
\begin{tabular}{l}
$\lambda$-block of size $2 {\times} 2$,  \\
and\\
$\mathcal{K}$-block of size $3 {\times} 3$  
\end{tabular}
&
$
\begin{matrix}
\lambda=-\frac{x_1}{a_1}
\end{matrix}
$
&
$
\begin{matrix}
x_1=0, \\
\mathrm{codim}\, S=1
\end{matrix}
$
&
$x_1, x_2, x_4$\\
& & & & & \\[-3ex]\hline

& & & & &\\
$\begin{matrix}A_{5,27}\\
(\mathrm{ind}\, = 1)
\end{matrix}
$
 &
$
\begin{tabular}{l}
$[e_2,e_3]=e_1$,\\
$[e_1,e_5]=e_1$,\\
$[e_3,e_5]=e_3+e_4$,\\
$[e_4,e_5]=e_1+e_4$
\end{tabular}
$ 
&
\begin{tabular}{l}
$\lambda$-block of size $2 {\times} 2$,  \\
 and\\
$\mathcal{K}$-block of size $3 {\times} 3$  
\end{tabular}
&
$
\begin{matrix}
\lambda=-\frac{x_1}{a_1}
\end{matrix}
$
&
$
\begin{matrix}
x_1=0, \\
\mathrm{codim}\, S=1
\end{matrix}
$
&
$x_1, x_2, x_4$\\
& & & & & \\[-3ex]\hline

& & & & &\\
$\begin{matrix}A_{5,28}^{a}\\
(\mathrm{ind}\, = 1)
\end{matrix}
$
 &
$
\begin{tabular}{l}
$[e_2,e_3]=e_1$,\\
$[e_1,e_5]=ae_1$,\\
$[e_2,e_5]=(a-1)e_2$,\\
$[e_3,e_5]=e_3+e_4$,\\
$[e_4,e_5]=e_4$
\end{tabular}
$ 
&
\begin{tabular}{l}
$\lambda$-block of size $2 {\times} 2$,  \\
and\\
$\mathcal{K}$-block of size $3 {\times} 3$  
\end{tabular}
&
$
\begin{matrix}
\lambda=-\frac{x_1}{a_1}
\end{matrix}
$
&
$
\begin{matrix}
x_1=0, \\
\mathrm{codim}\, S=1
\end{matrix}
$
&
$x_1, x_2, x_4$\\
& & & & & \\[-3ex]\hline

& & & & &\\
$\begin{matrix}A_{5,29}\\
(\mathrm{ind}\, = 1)
\end{matrix}
$
 &
$
\begin{tabular}{l}
$[e_2,e_4]=e_1$,\\
$[e_1,e_5]=e_1$,\\
$[e_2,e_5]=e_2$,\\
$[e_4,e_5]=e_3$
\end{tabular}
$ 
&
\begin{tabular}{l}
$\lambda$-block of size $4 {\times} 4$,  \\
and\\
$\mathcal{K}$-block of size $1 {\times} 1$  
\end{tabular}
&
$
\begin{matrix}
\lambda=-\frac{x_1}{a_1}
\end{matrix}
$
&
$
\begin{matrix}
x_1=0, \\
\mathrm{codim}\, S=1
\end{matrix}
$
&
$x_1, x_2, x_3$\\
& & & & & \\[-3ex]\hline

& & & & &\\
$\begin{matrix}A_{5,30}^{a}\\
(\mathrm{ind}\, = 1)
\end{matrix}
$
 &
$
\begin{tabular}{l}
$[e_2,e_4]=e_1$,\\
$[e_3,e_4]=e_2$,\\
$[e_1,e_5]=(a+1)e_1$,\\
$[e_2,e_5]=ae_2$,\\
$[e_3,e_5]=(a-1)e_3$,\\
$[e_4,e_5]=e_4$
\end{tabular}
$ 
&
\begin{tabular}{l}
$\mathcal{K}$-block of size $5 {\times} 5$  
\end{tabular}
&

&
$
\begin{matrix}
x_1=0, \\
x_2=0,\\
\mathrm{codim}\, S=2
\end{matrix}
$
&
$x_1, x_2, x_3$\\\hline
\end{tabular}

\newpage


$\quad$

\vskip-40pt
\hskip-15pt
\footnotesize{
\begin{tabular}{|c| c| c| c| c| c|}\hline
 \begin{tabular}{l}\textbf{Name} \\ \textbf{and} \\ \textbf{Index} \end{tabular}  & \textbf{Relations} & \begin{tabular}{l}\textbf{Jordan--Kronecker} \\  \textbf{invariant} \end{tabular} & \begin{tabular}{l} \textbf{Char.}\\ \textbf{number} \end{tabular}  & \textbf{Singular set}  &  \textbf{Family $\mathcal{G}_a$}\\\hline\hline

& & & & &\\
$\begin{matrix}A_{5,31}\\
(\mathrm{ind}\, = 1)
\end{matrix}
$
 &
$
\begin{tabular}{l}
$[e_2,e_4]=e_1$,\\
$[e_3,e_4]=e_2$,\\
$[e_1,e_5]=3e_1$,\\
$[e_2,e_5]=2e_2$,\\
$[e_3,e_5]=e_3$,\\
$[e_4,e_5]=e_3+e_4$
\end{tabular}
$ 
&
\begin{tabular}{l}
$\mathcal{K}$-block of size $5 {\times} 5$  
\end{tabular}
&

&
$
\begin{matrix}
x_1=0, \\
x_2=0,\\
\mathrm{codim}\, S=2
\end{matrix}
$
&
$x_1, x_2, x_3$\\
& & & & & \\[-2ex]\hline

& & & & &\\
$\begin{matrix}A_{5,32}^{a}\\
(\mathrm{ind}\, = 1)
\end{matrix}
$
 &
$
\begin{tabular}{l}
$[e_2,e_4]=e_1$,\\
$[e_3,e_4]=e_2$,\\
$[e_1,e_5]=e_1$,\\
$[e_2,e_5]=e_2$,\\
$[e_3,e_5]=ae_1+e_3$
\end{tabular}
$ 
&
\begin{tabular}{l}
$\mathcal{K}$-block of size $5 {\times} 5$  
\end{tabular}
&

&
$
\begin{matrix}
x_1=0, \\
x_2=0,\\
\mathrm{codim}\, S=2
\end{matrix}
$
&
$x_1, x_2, x_3$\\
& & & & & \\[-2ex]\hline

& & & & &\\
$\begin{matrix}A_{5,33}^{ab}\\
(\mathrm{ind}\, = 1)
\end{matrix}
$
 &
$
\begin{tabular}{l}
$[e_1,e_4]=e_1$,\\
$[e_3,e_4]=be_3$,\\
$[e_2,e_5]=e_2$,\\
$[e_3,e_5]=ae_3$,\\
$(a^2+b^2\ne0)$
\end{tabular}
$ 
&
\begin{tabular}{l}
$\mathcal{K}$-block of size $5 {\times} 5$  
\end{tabular}
&

&
$
\begin{matrix}
x_1x_2=0, \\
x_1x_3=0,\\
x_2x_3=0,\\
\mathrm{codim}\, S=3
\end{matrix}
$
&
$x_1, x_2, x_3$\\
& & & & & \\[-2ex]\hline

& & & & &\\
$\begin{matrix}A_{5,34}^{a}\\
(\mathrm{ind}\, = 1)
\end{matrix}
$
 &
$
\begin{tabular}{l}
$[e_1,e_4]=ae_1$,\\
$[e_2,e_4]=e_2$,\\
$[e_3,e_4]=e_3$,\\
$[e_1,e_5]=e_1$,\\
$[e_3,e_5]=e_2$
\end{tabular}
$ 
&
\begin{tabular}{l}
$\mathcal{K}$-block of size $5 {\times} 5$  
\end{tabular}
&

&
$
\begin{matrix}
x_2=0, \\
x_1x_3=0,\\
\mathrm{codim}\, S=2
\end{matrix}
$
&
$x_1, x_2, x_3$\\
& & & & & \\[-2ex]\hline

& & & & &\\
$\begin{matrix}A_{5,35}^{ab}\\
(\mathrm{ind}\, = 1)
\end{matrix}
$
 &
$
\begin{tabular}{l}
$[e_1,e_4]=be_1$,\\
$[e_2,e_4]=e_2$,\\
$[e_3,e_4]=e_3$,\\
$[e_1,e_5]=ae_1$,\\
$[e_2,e_5]=-e_3$,\\
$[e_3,e_5]=e_2$,\\
$(a^2+b^2\ne0)$
\end{tabular}
$ 
&
\begin{tabular}{l}
$\mathcal{K}$-block of size $5 {\times} 5$  
\end{tabular}
&

&
$
\begin{matrix}
x_1x_2=0, \\
x_2^2+x_3^2=0,\\
\mathrm{codim}\, S=2
\end{matrix}
$
&
$x_1, x_2, x_3$\\\hline

\end{tabular}

\newpage


$\quad$

\vskip-40pt

\hskip-47pt 
\footnotesize{\begin{tabular}{|c| c| c| c| c| c|}\hline
 \begin{tabular}{l}\textbf{Name} \\ \textbf{and} \\ \textbf{Index} \end{tabular}  & \textbf{Relations} & \begin{tabular}{l}\textbf{Jordan--Kronecker} \\  \textbf{invariant} \end{tabular} & \begin{tabular}{l} \textbf{Char.}\\ \textbf{number} \end{tabular}  & \textbf{Singular set}  &  \textbf{Family $\mathcal{G}_a$}\\\hline\hline

& & & & &\\
$\begin{matrix}A_{5,36}\\
(\mathrm{ind}\, = 1)
\end{matrix}
$
 &
$
\hskip-10pt\begin{tabular}{l}
$[e_2,e_3]=e_1$,\\
$[e_1,e_4]=e_1$,\\
$[e_2,e_4]=e_2$,\\
$[e_2,e_5]=-e_2$,\\
$[e_3,e_5]=e_3$
\end{tabular}
$ 
&
\hskip-10pt\begin{tabular}{l}
$\mathcal{K}$-block of size $5 {\times} 5$  
\end{tabular}
&
\hskip-10pt
&
$
\hskip-10pt\begin{matrix}
x_1=0, \\
x_2x_3=0,\\
\mathrm{codim}\, S=2
\end{matrix}
$
&
$
\hskip-20pt\begin{array}{c}
\frac{1}{a_1^2}(a_1^2x_5+a_1a_3x_2+\\
+a_1a_2x_3-a_2a_3x_1),\\\\
\frac{1}{a_1^3}(a_1^2x_2x_3-a_1a_3x_1x_2-\\
-a_1a_2x_1x_3+a_2a_3x_1^2),\\\\
\frac{1}{a_1^4}(-a_1^2x_1x_2x_3+a_1a_3x_1^2x_2+\\
+a_1a_2x_1^2x_3-a_2a_3x_1^3)
\end{array}
$\\\hline

$\begin{matrix}A_{5,37}\\
(\mathrm{ind}\, = 1)
\end{matrix}
$ &
$
\hskip-10pt\begin{tabular}{l}
$[e_2,e_3]=e_1$,\\
$[e_1,e_4]=2e_1$,\\
$[e_2,e_4]=e_2$,\\
$[e_3,e_4]=e_3$,\\
$[e_2,e_5]=-e_3$,\\
$[e_3,e_5]=e_2$
\end{tabular}
$ 
&
\hskip-10pt\begin{tabular}{l}
$\mathcal{K}$-block of size $5 {\times} 5$  
\end{tabular}
&
\hskip-10pt
&
$
\hskip-10pt\begin{matrix}
x_1=0, \\
x_2^2+x_3^2=0,\\
\mathrm{codim}\, S=2
\end{matrix}
$
&
$
\hskip-10pt\begin{array}{c}
\frac{1}{a_1^2}\bigl(-(a_2^2+a_3^2)x_1+2a_1a_2x_2+\\
+2a_1a_3x_3+2a_1^2x_5\bigr),\\\\
\frac{1}{a_1^3}\bigl((a_2^2+a_3^2)x_1^2-2a_1a_2x_1x_2-\\
-2a_1a_3x_1x_3+a_1^2x_2^2+a_1^2x_3^2\bigr),\\\\
\frac{1}{a_1^4}\bigl(-(a_2^2+a_3^2)x_1^3+2a_1a_2x_1^2x_2+\\
+2a_1a_3x_1^2x_3-a_1^2x_1x_2^2-a_1^2x_1x_3^2\bigr)
\end{array}
$\\\hline

$\begin{matrix}A_{5,38}\\
(\mathrm{ind}\, = 1)
\end{matrix}
$
 &
$
\hskip-10pt\begin{tabular}{l}
$[e_1,e_4]=e_1$,\\
$[e_2,e_5]=e_2$,\\
$[e_4,e_5]=e_3$
\end{tabular}
$ 
&
\hskip-10pt\begin{tabular}{l}
$\lambda_1$-block of size $2 {\times} 2$,  \\
$\lambda_2$-block of size $2  {\times} 2$,  \\
$\lambda_1\ne \lambda_2$ \\
 and\\
$\mathcal{K}$-block of size $1 {\times} 1$  
\end{tabular}
&
$
\hskip-10pt\begin{matrix}
\lambda_1=-\frac{x_1}{a_1}\\
\lambda_2=-\frac{x_2}{a_2}
\end{matrix}
$
&
$
\hskip-10pt\begin{matrix}
x_1x_2=0, \\
\mathrm{codim}\, S=1
\end{matrix}
$
&
$x_1, x_2, x_3$\\\hline

$\begin{matrix}A_{5,39}\\
(\mathrm{ind}\, = 1)
\end{matrix}
$
 &
$
\hskip-10pt\begin{tabular}{l}
$[e_1,e_4]=e_1$,\\
$[e_2,e_4]=e_2$,\\
$[e_1,e_5]=-e_2$,\\
$[e_2,e_5]=e_1$,\\
$[e_4,e_5]=e_3$
\end{tabular}
$ 
&
\hskip-10pt\begin{tabular}{l}
$\lambda_1$-block of size $2 {\times} 2$,  \\
$\lambda_2$-block of size $2  {\times} 2$,  \\
$\lambda_1\ne \lambda_2$ \\
 and\\
$\mathcal{K}$-block of size $1 {\times} 1$  
\end{tabular}
&
$
\hskip-5pt\begin{matrix}
\lambda_1=-\frac{x_1+ix_2}{a_1+ia_2}\\
\lambda_2=-\frac{x_1-ix_2}{a_1-ia_2}
\end{matrix}
$
&
$
\hskip-10pt\begin{matrix}
x_1^2+x_2^2=0, \\
\mathrm{codim}\, S=1
\end{matrix}
$
&
$x_1, x_2, x_3$\\\hline

$\begin{matrix}A_{5,40}\\
(\mathrm{ind}\, = 1)
\end{matrix}
$
 &
$
\hskip-10pt\begin{tabular}{l}
$[e_1,e_2]=2e_1$,\\
$[e_1,e_3]=-e_2$,\\
$[e_2,e_3]=2e_3$,\\
$[e_1,e_4]=e_5$,\\
$[e_2,e_4]=e_4$,\\
$[e_2,e_5]=-e_5$,\\
$[e_3,e_5]=e_4$
\end{tabular}
$ 
&
\hskip-10pt\begin{tabular}{l}
$\mathcal{K}$-block of size $5 {\times} 5$  
\end{tabular}
&
\hskip-10pt
&
$
\hskip-10pt\begin{matrix}
x_4=0, \\
x_5=0,\\
\mathrm{codim}\, S=2
\end{matrix}
$
&
$
\hskip-20pt\begin{array}{c}
x_1x_4^2-x_2x_4x_5-x_3x_5^2,\\\\
2a_4x_1x_4+a_1x_4^2-a_5x_2x_4-\\
-a_4x_2x_5-a_2x_4x_5-\\
-2a_5x_3x_5-a_3x_5^2,\\\\
a_4^2x_1-a_4a_5x_2-a_5^2x_3+\\
+(2a_1a_4-a_2a_5)x_4-\\
-(a_2a_4+2a_3a_5)x_5
\end{array}
$\\\hline
\end{tabular}}

\end{document}